\documentclass[12pt, a4paper]{extarticle}

\RequirePackage[OT1]{fontenc}
\RequirePackage{amsthm,amsmath}
\usepackage{amsthm, amssymb}
\RequirePackage[numbers]{natbib}
\RequirePackage[colorlinks,citecolor=blue,urlcolor=blue]{hyperref}

\usepackage{algorithm}
\usepackage{algorithmicx}
\usepackage{algpseudocode}
\usepackage{ThesisShortcuts}
\usepackage{graphicx}
\usepackage{placeins}

\newtheorem{theorem}{Theorem}[section]
\newtheorem{lemma}[theorem]{Lemma}
\newtheorem{proposition}[theorem]{Proposition}
\newtheorem{definition}[theorem]{Definition}

\newcommand{\indc}{\displaystyle{1\!\!1_{}}}
\newcommand{\Esp}{\mathbb{E}}
\newcommand{\ds}{\displaystyle}

\newcommand{\frc}[2]{\frac{\ds #1}{\ds #2}}
\newcommand{\ens}{\enspace}

\newcommand{\Prb}{\mathbb{P}}

\newcommand{\hLMMD}{\hat{\Delta}}
\newcommand{\hDn}{\hat{\Delta}}

\newcommand{\Reals}{\mathbb{R}}

\newcommand{\underoverset}[3]{\underset{#1}{\overset{#2}{#3}}}
\newcommand{\hsp}{\mathcal{H}}
\newcommand{\bk}{\bar{k}}
\newcommand{\bmu}{\bar{\mu}}
\newcommand{\tm}{\tilde{m}}
\newcommand{\tsig}{\tilde{\Sigma}}
\newcommand{\hm}{\hat{m}}
\newcommand{\hsig}{\hat{\Sigma}}
\newcommand{\NN}[2]{N[ #1, #2]}
\newcommand{\TT}{\mathcal{T}}
\newcommand{\NNoT}[2]{(N o \TT)[ #1, #2]}
\newcommand{\NNoTbis}[1]{(N o \TT)[ #1 ]}

\newcommand{\hypot}[1]{\mathbf{H}_{#1}}
\newcommand{\hbmu}{\hat{\bar{\mu}}_P}
\newcommand{\llim}[2]{\underset{#1 \to #2}{\longrightarrow}}

\algrenewcommand\algorithmicrequire{\textbf{Input:}}
\algrenewcommand\algorithmicreturn{\textbf{Output:}}

\newcommand{\prodscal}[2]{\langle #1 , #2 \rangle} 

\begin{document}

\title{A One-Sample Test for Normality with Kernel Methods }
\date{}



\author{\textsc{J\'er\'emie Kellner}\\
		\small Laboratoire de Math\'ematiques UMR 8524 CNRS - Universit\'e Lille 1 - MODAL team-project Inria\\
   \small \texttt{jeremie.kellner@ed.univ-lille1.fr}\\
   		\textsc{Alain Celisse}\\
   		\small Laboratoire de Math\'ematiques UMR 8524 CNRS - Universit\'e Lille 1 - MODAL team-project Inria\\
   \small \texttt{celisse@math.univ-lille1.fr}}

\maketitle

%




\begin{abstract}
We propose a new one-sample test for normality in a Reproducing Kernel Hilbert Space (RKHS). Namely, we test the null-hypothesis of belonging to a given family of Gaussian distributions. Hence our procedure may be applied either to test data for normality or to test parameters (mean and covariance) if data are assumed Gaussian.
Our test is based on the same principle as the MMD (Maximum Mean Discrepancy) which is usually used for two-sample tests such as homogeneity or independence testing. Our method makes use of a special kind of parametric bootstrap (typical of goodness-of-fit tests) which is computationally more efficient than standard parametric bootstrap. 
Moreover, an upper bound for the Type-II error highlights the dependence on influential quantities. Experiments illustrate the practical improvement allowed by our test in high-dimensional settings where common normality tests are known to fail. We also consider an application to covariance rank selection through a sequential procedure.
\end{abstract}

%
\tableofcontents

\section{Introduction}
Non-vectorial data such as DNA sequences or pictures often require a positive semi-definite kernel \cite{Aronszajn} which plays the role of a similarity function. For instance, two strings can be compared by counting the number of common substrings. 
Further analysis is then carried out in the associated reproducing kernel Hilbert space (RKHS), that is the Hilbert space spanned by the evaluation functions $k(x, .)$ for every $x$ in the input space. Thus embedding data into this RKHS through the feature map $x \mapsto k(x, .)$ allows to apply linear algorithms to initially non-vectorial inputs. 

Embedded data are often assumed to have a Gaussian distribution.
For instance supervised and unsupervised classification are performed in \cite{Bouveyron} by modeling each class as a Gaussian process. 
In \cite{Roth2006}, outliers are detected by modelling embedded data as a Gaussian random variable and by removing points lying in the tails of that Gaussian distribution.
This key assumption is also made in \cite{SrivastavaHiDimMeanTest} where a mean equality test is used in high-dimensional setting.
Moreover, Principal Component Analysis (PCA) and its kernelized version Kernel PCA \cite{KernelPCA} are known to be optimal for Gaussian data as these methods rely on second-order statistics (covariance). Besides, a Gaussian assumption allows to use Expectation-Minimization (EM) techniques to speed up PCA \cite{Roweis1998}. 
%

Depending on the (finite or infinite dimensional) structure of the RKHS, Cramer-von-Mises-type normality tests can be applied, such as Mardia's skewness test \cite{MardiaSkewKurt}, the Henze-Zirkler test \cite{HenzeZirkler} and the Energy-distance test \cite{SzekelyRizzoENormTest}. 
However these tests become less powerful as dimension increases (see Table 3 in \cite{SzekelyRizzoENormTest}). 
An alternative approach consists in randomly projecting high-dimensional objects on one-dimensional directions and then applying univariate test on a few randomly chosen marginals \cite{RandomProjTest}. This projection pursuit method has the advantage of being suited to high-dimensional settings. On the other hand, such approaches also suffer a lack of power because of the limited number of considered directions (see Section 4.2 in \cite{RandomProjTest}).

In the RKHS setting, \cite{Gretton_2007} introduced the Maximum Mean Discrepancy (MMD) which quantifies the gap between two distributions through distances between two elements of an RKHS. The MMD approach has been used for two-sample testing \cite{Gretton_2007} and for independence testing (Hilbert Space Independence Criterion, \cite{HSIC_2007}). However to the best of our knowledge, MMD has not been applied in a one-sample goodness-of-fit testing framework.

The main contribution of the present paper is to provide a one-sample statistical test of normality for data in a general Hilbert space (which can be an RKHS), by means of the MMD principle. This test features two possible applications: testing the normality of the data but also testing parameters (mean and covariance) if data are assumed Gaussian. The latter application encompasses many current methods that assume normality to make inferences on parameters, for instance to test the nullity of the mean \cite{SrivastavaHiDimMeanTest} or to assess the sparse structure of the covariance \citep{Svantesson2003, Bien2011}.

Once the test statistic is defined, a critical value is needed to decide whether to accept or reject the Gaussian hypothesis. In goodness-of-fit testing, this critical value is typically estimated by parametric bootstrap. Unfortunately, parametric bootstrap requires parameters to be computed several times, hence heavy computational costs (\textit{i.e.} diagonalization of covariance matrices). Our test bypasses the recomputation of parameters by implementing a faster version of parametric bootstrap. Following the idea of \cite{Kojadinovic2012}, this fast bootstrap method "linearizes" the test statistic through a Fr\'echet derivative approximation and thus can estimate the critical value by a \textit{weighted} bootstrap (in the sense of \cite{Burke2000}) which is computationally more efficient. Furthermore our version of this bootstrap method allows parameters estimators that are not explicitly "linear" (\textit{i.e.} that consist of a sum of independent terms) and that take values in possible infinite-dimensional Hilbert spaces.

Finally, we illustrate our test and present a sequential procedure that assesses the rank of a covariance operator. 
The problem of covariance rank estimation is adressed in several domains: functional regression \citep{Cardot2010, Brunel2013}, classification \cite{ZwaldKPCA} and dimension reduction methods such as PCA, Kernel PCA and Non-Gaussian Component Analysis \cite{NGCA_2006, Diederichs2010, Diederichs2013} where the dimension of the kept subspace is a crucial problem.  

Here is the outline of the paper. Section~\ref{ssec.framework} sets our framework and Section~\ref{sec.MMDTest} introduces the MMD and how it is used for our one-sample test.
The new normality test is described in Section~\ref{sec.new.test}, while both its theoretical and empirical performances are detailed in Section~\ref{sec.theoretical.assessment} in terms of control of Type-I and Type-II errors.
A sequential procedure to select covariance rank is presented in Section~\ref{sec.covranksel}.
%
%
\section{Framework}
\label{ssec.framework}
\sloppy 

Let $(\mathcal{H},\mathcal{A})$ be a measurable space, and $Y_1, \dots, Y_n \in \mathcal{H}$ denote a sample of \emph{independent and identically distributed} (\textit{i.i.d.}) random variables drawn from an unknown distribution $P\in\mathcal{P}$, where  $\mathcal{P}$ is a set of distributions defined on $\mathcal{A}$. 

In our framework, $\hsp$ is a separable Hilbert space endowed with a dot product $\langle~.~,~.~\rangle_\hsp$ and the associated norm $\Vert . \Vert_\hsp$ (defined by $\Vert h \Vert_{\hsp} = \langle h, h \rangle^{1/2}_\hsp$ for any $h \in \hsp$). Our goal is to test whether $Y_i$ is a \textit{Gaussian random variable (r.v.)} of $\hsp$, which is defined as follows.
\begin{definition}{(Gaussian random variable in a Hilbert space)}\\
\label{DefGauss1}
Let $( \Omega, \mathcal{F}, \mathbb{P} )$ a measure space, $(\mathcal{H}, \mathcal{F}^{'})$ a measurable space where $\mathcal{H}$ is a Hilbert space, and $Y: \Omega \to \mathcal{H}$ a measurable map.\\
$Y$ is a Gaussian r.v. of $\mathcal{H}$ if  $\langle Y, h \rangle_\hsp$ is a univariate Gaussian r.v. for any $h \in \mathcal{H}$.\\
Assuming that $\Esp_Y ||Y||_{\hsp} < +\infty$, there exists $m \in \mathcal{H}$ such that:
\begin{equation*}
\forall h \in \mathcal{H}, \enspace \langle m, h\rangle _\mathcal{H} = \Esp_Y \langle Y, h\rangle _\mathcal{H} \ens ,
\end{equation*}
and a (finite trace) operator $\Sigma: \mathcal{H} \to \mathcal{H}$ satisfying:
\begin{equation*}
\forall h, h^{'} \in \mathcal{H}, \enspace \langle \Sigma h, h^{'}\rangle _\mathcal{H} = \mathrm{cov} (\langle Y, h\rangle _\mathcal{H}, \langle Y, h^{'}\rangle _\mathcal{H}) \ens .
\end{equation*}
$m$ and $\Sigma$ are respectively the mean and the covariance operator of $Y$. The distribution of $Y$ is denoted $\mathcal{N}(m, \Sigma)$.
\end{definition}
%
%
%
More precisely, the tested hypothesis is that $Y_i$ follows a Gaussian distribution $\mathcal{N}(m_0, \Sigma_0)$, where $(m_0, \Sigma_0) \in \Theta_0$ and $\Theta_0$ is a subset of the parameter space $\Theta$. 
\footnote{
The parameter space $\Theta$ is endowed with the dot product $\langle (m, \Sigma), (m^{'}, \Sigma^{'}) \rangle_{\Theta} = \langle m, m^{'} \rangle_{\hsp} + \langle \Sigma, \Sigma^{'} \rangle_{HS(\hsp)}$, where $HS(\hsp)$ is the space of Hilbert-Schmidt (finite trace) operators $\hsp \to \hsp$ and $\langle \Sigma, \Sigma^{'} \rangle_{HS(\hsp)} = \sum_{i \geq 1} \langle \Sigma e_i, \Sigma^{'} e_i \rangle_\hsp$ for any complete orthonormal basis $(e_i)_{i \geq 1}$ of $\hsp$. 
Therefore, for any $\theta \in \Theta$, the tensor product $\theta^{\otimes 2}$ is defined as the operator $\Theta \to \Theta, \theta^{'} \mapsto \langle \theta, \theta^{'} \rangle_{\Theta} \theta$. For any $\theta \in \Theta$ and $\bar{h} \in H(\bk)$, the tensor product $\bar{h} \otimes \theta$ is the operator $\Theta \to H(\bk), \theta^{'} \mapsto \langle \theta, \theta^{'} \rangle_{\Theta} \bar{h}$.
}
Following \cite{LehRom_2005}, let us define the null hypothesis $\hypot{0}:\ P\in \mathcal{P}_0$, and the alternative hypothesis $\hypot{1}:\ P\not\in \mathcal{P}\setminus \mathcal{P}_0$ where the subset of null-hypotheses $\mathcal{P}_0 \subseteq \mathcal{P}$ is
\begin{align*}
\mathcal{P}_0 = \{ \mathcal{N}(m_0, \Sigma_0) \mid (m_0, \Sigma_0) \in \Theta_0 \} \ens .
\end{align*}
The purpose of a statistical test $T(Y_1,\ldots,Y_n)$ of $\hypot{0}$ against $\hypot{1}$ is to distinguish between the null ($\hypot{0}$) and the alternative ($\hypot{1}$) hypotheses. It requires two elements: a statistic $n \hat{\Delta}^2$ (which we define in Section \ref{sssec.new.test.stat}) that measures the gap between the empirical distribution of the data and the considered family of normal distributions $\mathcal{P}_0$, and a rejection region $\mathcal{R}_\alpha$ (at a level of confidence $0 < \alpha < 1$). $\hypot{0}$ is accepted if and only if $n \hat{\Delta}^2 \notin \mathcal{R}_\alpha$. The rejection region is determined by the distribution of $n \hat{\Delta}^2$ under the null-hypothesis such that the probability of wrongly rejecting $\hypot{0}$ (Type-I error) is controlled by $\alpha$.


%
%

\section{The Maximum Mean Discrepancy (MMD)}
\label{sec.MMDTest}

Following \cite{Gretton_2007} the gap between two distributions $P$ and $Q$ on $\hsp$ is measured by
\begin{equation}
\label{GeneralGOFStat}
\Delta(P, Q) = \sup_{f \in \mathcal{F}} |\Esp_{Y \sim P} f(Y) - \Esp_{Z \sim Q} f(Z)|,
\end{equation}
where $\mathcal{F}$ is a class of real valued functions on $\hsp$. 
%
%
Regardless of $\mathcal{F}$, \eqref{GeneralGOFStat} always defines a pseudo-metric \footnote{ A pseudo-metric $\Delta(.,.)$ satisfies for any $P$, $Q$,$R$ :
(i) $\Delta(P,P) = 0$, (ii) $\Delta(P,Q) = \Delta(Q,P)$,
and (iii) $\Delta(P,R) \leq \Delta(P,Q) + \Delta(Q,R)$.}
on probability distributions \cite{HilbertEmbed}. %

The choice of $\mathcal{F}$ is subject to two requirements: \textit{(i)} \eqref{GeneralGOFStat} must define a metric between distributions, that is
\begin{equation}
\label{NeededProperty}
\forall P, Q , \ens \Delta(P,Q) = 0 \Rightarrow P = Q \ens ,
\end{equation}
and \textit{(ii)} \eqref{GeneralGOFStat} must be expressed in an easy-to-compute form (without the supremum term).

To solve those two issues, several papers \citep{Kernel2sample, FastConsistentKern2Test, HSIC_2007} have considered the case when $\mathcal{F}$ is the unit ball of a reproducing kernel Hilbert space (RKHS) $H(\bar{k})$ associated with a positive semi-definite kernel $\bar{k}: \hsp \times \hsp \to \Reals$.
\begin{definition}{(Reproducing Kernel Hilbert space, \cite{Aronszajn})}
Let $\bk$ be a positive semi-definite kernel, i.e.
\begin{align*}
\forall x_1, \ldots, x_n \in \hsp, \forall \alpha_1, \ldots, \alpha_n, \sum_{i, j=1}^n \alpha_i \alpha_j \bk(x_i, x_j) \geq 0 \ens ,
\end{align*}
with equality if and only if $\alpha_1 = \ldots = \alpha_n = 0$.\\
There exists a unique Hilbert space $H(\bk)$ of real-valued functions on $\hsp$ which satisfies:
\begin{itemize}
\item $\forall x \in \hsp, \ens \bk(x, .) \in H(\bk)$  ,
\item $\forall f \in \hsp, \ens \forall x \in \mathcal{X}, \enspace \langle f, \bk(x, .)\rangle _{H(\bk)} = f(x)$  .
\end{itemize}
$H(\bk)$ is the reproducing kernel Hilbert space (RKHS) of $\bk$.
\end{definition}

Let $\Vert . \Vert_{H(\bk)} = \langle ., . \rangle^{1/2}$ be the norm of $H(\bk)$ and $\mathcal{B}_1(\bk) = \{f \in H(\bk) \mid ||f||_{H(\bk)} \leq 1\}$ denote the unit ball of $H(\bk)$. When $\mathcal{F} = \mathcal{B}_1(\bk)$, $\Delta(\cdot,\cdot)$ becomes a metric only for a class of kernels $\bk$ that are called  \emph{characteristic}.
\begin{definition}{(Characteristic kernel, \cite{Fukumizu2007_introchar})}\\
Let $\mathcal{F} = \mathcal{B}_1(\bk)$ in \eqref{GeneralGOFStat} for some kernel $\bk$. 
Then $\bk$ is a characteristic kernel if $\Delta(P, Q)=0$ implies $P=Q$.
\end{definition}

Most common kernels are characteristic: Gaussian kernels $\bk(x, y) = \exp( - \sigma \Vert x - y \Vert_{\hsp}^2 )$ where $\sigma > 0$, the exponential kernel $\bk(x, y) = \exp(\langle x, y \rangle_{\hsp})$ and Student kernels $\bk(x, y) = (1 + \sigma ||x - y||^2_{\hsp})^{-\alpha}$ where $\alpha, \sigma > 0$, to name a few. Several criteria for a kernel to be characteristic exist (see \citep{HilbertEmbed, CharKernGroups, Christmann2010}).
%

%
Moreover taking $\mathcal{F} = \mathcal{B}_1(\bk)$ enables to cast $\Delta(P,Q)$ as an easy to compute quantity. This is done by embedding any distribution $P$ in the RKHS $H(\bar{k})$ as follows.
\begin{definition}{(Hilbert space embedding, Lemma 3 from \cite{Kernel2sample})}
Let $P$ be a distribution such that $\Esp_{Y \sim P} \bk^{1/2}(Y, Y)<+\infty$.\\
Then there exists $\bmu_P \in H(\bk)$ such that for every $f \in H(\bk)$,
\begin{align*}
\langle \bmu_P, f \rangle _{H(\bk)} = \Esp f(Y) \ens .
\end{align*}
$\bmu_P$ is called the Hilbert space embedding of $P$ in $H(\bk)$.
\end{definition}
Thus $\Delta(P, Q)$ can be expressed as the gap between the Hilbert space embeddings of $P$ and $Q$ (Lemma~$4$ in \cite{Kernel2sample}):
\begin{align}
\Delta (P,Q) & =  \sup_{f \in H(\bk), ||f||_{H(\bk)} \leq 1} |\Esp_P f(Y) - \Esp_{Q} f(Z)|\notag\\
         & =  \sup_{f \in H(\bk), ||f||_{H(\bk)} \leq 1} |\langle \bmu_{P} - \bmu_{Q}, f\rangle _{H(\bk)}|\notag\\
         & =  ||\bmu_{P} - \bmu_{Q}||_{H(\bk)} \ens .
         \label{MMDGap}
\end{align}
\eqref{MMDGap} is called the Maximum Mean Discrepancy (MMD) between $P$ and $Q$.
%

%
%

Within our framework the goal is to compare $P$ the true distribution of the data with a Gaussian distribution $P_0 = \mathcal{N}(m_0, \Sigma_0)$ for some $(m_0, \Sigma_0) \in \Theta_0$. Hence the quantity of interest is
\begin{align}
\Delta^2 = \left\Vert \bar{\mu}_{P} - \bar{\mu}_{P_0} \right\Vert^2_{H(\bk)} \ens . 
\label{norm.mmd}
\end{align}
For the sake of simplicity, we use the notation
\begin{align*}
\bar{\mu}_{\mathcal{N}(m, \Sigma)} = \NN{m}{\Sigma}
\end{align*}
to denote the Hilbert space embedding of a Gaussian distribution.

%
%
%

\section{Kernel normality test}
\label{sec.new.test}

This section introduces our one-sample test for normality based on the quantity \eqref{norm.mmd}. 
As said in Section~\ref{ssec.framework}, we test the null-hypothesis $\hypot{0}: P \in \{ \mathcal{N}(m_0, \Sigma_0) \mid (m_0, \Sigma_0) \in \Theta_0 \}$ where $\Theta_0$ is a subset of the parameter space. Therefore our procedure may be used as test for normality or a test on parameter if data are assumed Gaussian. The test procedure is summed up in Algorithm~\ref{sumup.test}.
\begin{algorithm}[t]
\caption{Kernel Normality Test procedure}
\label{sumup.test}
\begin{algorithmic}
\normalsize
\Require{$Y_1, \dots, Y_n \in \hsp$, $\bk: \hsp \times \hsp \to \R$ (kernel) and $0<\alpha<1$ (test level).}
\begin{enumerate}
\item Compute $K = \croch{\langle Y_i, Y_j \rangle}_{i, j}$  (Gram matrix).
\item Compute $n \hLMMD^2$ (test statistic) from \eqref{gofstat} that depends on $K$ and $\bk$ (Section \ref{sssec.new.test.stat})
\item 
\begin{enumerate}
  \item Draw $B$ (approximate) independent copies of $n \hLMMD^2$ under $\hypot{0}$ by fast parametric bootstrap (Section \ref{sssc.qtl.estim}).
  \item Compute $\hat{q}_{\alpha, n}$ ($1-\alpha$ quantile of $n \hLMMD^2$ under $\hypot{0}$) from these replications.
\end{enumerate}
\end{enumerate}
\Return{ Reject $\hypot{0}$ if $n \hLMMD^2 > \hat{q}_{\alpha, n}$, and accept otherwise.}
\end{algorithmic}
\end{algorithm}

\subsection{Test statistic}
\label{sssec.new.test.stat}
As in \cite{Gretton_2007}, $\Delta^2$ can be estimated by replacing $\bar{\mu}_P $ with the sample mean 
\begin{align*}
\hbmu = \bar{\mu}_{\hat{P}} = (1/n) \sum_{i=1}^n \bar{k}(Y_i, . ) \ens ,
\end{align*} 
where $\hat{P} = (1/n) \sum_{i=1}^n \delta_{Y_i}$ is the empirical distribution. The null-distribution embedding $\NN{m_0}{\Sigma_0}$ is estimated by $\NN{\tilde{m}}{\tilde{\Sigma}}$ where $\tilde{m}$ and $\tilde{\Sigma}$ are appropriate and consistent (under $\hypot{0}$) estimators of $m_0$ and $\Sigma_0$. 
This yields the estimator 
\begin{align*}
\hLMMD^2 = \Vert \hbmu - \NN{\tilde{m}}{\tilde{\Sigma}} \Vert_{H(\bk)}^2 \ens ,
\end{align*}
which can be explicited by expanding the square of the norm and using the reproducing property of $H(\bk)$ as follows
\begin{align}
\hLMMD^2  = \frac{1}{n^2} \sum_{i, j = 1}^n \bar{k}(Y_i,Y_j)  - \frac{2}{n} \sum_{i = 1}^n \NN{\tilde{m}}{ \tilde{\Sigma}}(Y_i)  + \Vert \NN{\tilde{m}}{ \tilde{\Sigma}} \Vert_{H(\bk)}^2 \ens .
\label{gofstat}
\end{align}
Proposition~\ref{prop.def.L} ensures the consistency of the statistic  \eqref{gofstat}.   
%
\begin{proposition}\label{prop.def.L}
Assume that $P$ is Gaussian $\mathcal{N}(m_0, \Sigma_0)$ where $(m_0, \Sigma_0) \in \Theta_0$ and $(\tm, \tsig)$ are consistent estimators of $(m_0, \Sigma_0)$. Also assume that $\Esp_P \bar{k}(Y, Y)~<~+~\infty$ and $\NN{m}{\Sigma}$ is a continuous function of $(m, \Sigma)$ on $\Theta_0$.
Then $\hLMMD^2$ is a consistent estimator of $\Delta^2$. 
\end{proposition}
\begin{proof}
First, note that $\bmu_P$ exists since $\Esp \bk(Y, Y)<+\infty$ implies $\Esp \bk^{1/2}(Y, Y)<+\infty$. By the Law of Large Numbers in Hilbert Spaces \cite{HoffmannJorgensen1976}, $\hbmu \llim{n}{\infty} \bmu_P$ $P$-almost surely since $\Esp \Vert \bk(Y,\cdot) - \bmu_P \Vert^2_{H(\bk)} = \Esp \bk(Y,Y) - \Esp \bk(Y, Y^{'}) \leq \Esp \bk(Y,Y) + \Esp^2 \bk(Y, Y^{'}) < +\infty$. The continuity of $\NN{m}{\Sigma}$ (with respect to $(m, \Sigma)$) and the consistency of $(\tilde{m}, \tilde{\Sigma})$ yield $\NN{\tilde{m}}{\tilde{\Sigma}} \overset{P-a.s.}{\llim{n}{\infty}} \NN{m_0}{\Sigma_0}$ $P$-a.s.. Finally, the continuity of $\Vert \cdot \Vert^2_\hsp$ leads to $\hLMMD^2 \overset{P-a.s.}{\llim{n}{\infty}} \Delta^2$.
\end{proof}
The expressions for $\NN{\tilde{m}}{\tilde{\Sigma}}(Y_i)$ and $\Vert \NN{\tilde{m}}{\tilde{\Sigma}} \Vert_{H(\bk)}^2$ in \eqref{gofstat} depend on the choice of $\bar{k}$. Those are given by Propositions \ref{prop.teststat.gauss} and \ref{prop.teststat.exp} when $\bk$ is Gaussian and exponential. Note that in these cases, the continuity assumption of $\NN{m}{\Sigma}$ required by Proposition~\ref{prop.def.L} is satisfied.

Before stating Propositions \ref{prop.teststat.gauss} and \ref{prop.teststat.exp}, the following notation is introduced.
For a symmetric operator $L : \hsp \to \hsp$ with eigenexpansion $L = \sum_{r \geq 1} \lambda_r \Psi^{\otimes 2}$, its determinant is denoted $|L| = \prod_{r \geq 1} \lambda_r$. For any $q \in \Reals$, the operator $L^q$ is defined as $L^q = \sum_{r \geq 1} \lambda^q_r \indc_{\{\lambda_r > 0\}} \Psi_r^{\otimes 2}$.
\begin{proposition}{(Gaussian kernel case)}
\label{prop.teststat.gauss}
Let $\bk(.,.) = \exp(-\sigma \Vert .-. \Vert^2_\hsp)$ where $\sigma > 0$. Then, 
\begin{align}
\NN{\tilde{m}}{\tilde{\Sigma}}(.) & = |I + 2 \sigma \tilde{\Sigma}|^{-1/2} \exp\left( - \sigma ||(I + 2 \sigma \tilde{\Sigma})^{-1/2} (. - \tilde{m}) ||_\hsp^2 \right) \ens ,\notag \\
\Vert \NN{\tilde{m}}{\tilde{\Sigma}} \Vert_{H(\bk)}^2 & = |I + 4 \sigma \tilde{\Sigma} |^{-1/2} \notag \ens .
\end{align}
\end{proposition}
\begin{proposition}{(Exponential kernel case)}
\label{prop.teststat.exp}
Let $\bk(.,.) = \exp(\langle  ., . \rangle _\hsp)$. Assume that the largest eigenvalue of $\tilde{\Sigma}$ is smaller than $1$. Then, 
\begin{align}
\NN{\tilde{m}}{\tilde{\Sigma}}(.) & =  \exp\left( \langle \tilde{m}, .\rangle _\hsp +  \frac{1}{2} \langle \tilde{\Sigma} . , .\rangle _\hsp \right) \ens , \notag \\
||\NN{\tilde{m}}{\tilde{\Sigma}} ||^2 & = |I - \tilde{\Sigma}^2 |^{-1/2} \exp\left(\Vert(I - \tilde{\Sigma}^2)^{-1/2} \tilde{m}  \Vert_\hsp^2\right) \ens . \notag
\end{align}
\end{proposition}
The proofs of Propositions \ref{prop.teststat.gauss} and \ref{prop.teststat.exp} are provided in Appendix \ref{appendix.proof.prop.teststat}.

%
For most estimators $(\tm, \tsig)$, the quantities provided in Propositions \ref{prop.teststat.exp} and \ref{prop.teststat.gauss} are computable via the Gram matrix $K = \left[ \langle Y_i, Y_j \rangle_\hsp \right]_{1 \leq i, j \leq n}$. For instance, asumme that $(\tm, \tsig)$ are the classical estimators $(\hm, \hsig)$ where $\hm = (1/n) \sum_{i=1}^n Y_i$ and $\hsig = (1/n) \sum_{i=1}^n (Y_i - \hm)^{\otimes 2}$. Let $I_n$ and $J_n$ be respectively the $n \times n$ identity matrix and the $n \times n$ matrix whose all entries equal $1$, $H = I_n - (1/n) J_n$ and $K_c = H K H$ be the centered Gram matrix. Then for any $\Box \in \Reals$,
\begin{align}
\left| I + \Box \hsig  \right| = \det\left( I_n + \frac{\Box}{n} K_c    \right) \ens ,
\notag 
\end{align}
where $\det(.)$ denotes the determinant of a matrix and
\begin{align}
\left\Vert  (I + \Box \hsig)^{-1/2} Y_i \right\Vert^2_\hsp = \left[(I_n + \frac{\Box}{n} K_c)^{-1} \right]_{i,i} \ens , 
\notag 
\end{align}
where $[.]_{ii}$ denotes the entry in the $i$-th row and the $i$-th column of a matrix.

\subsection{Estimation of the critical value}
\label{sssc.qtl.estim}

Designing a test with confidence level $0 < \alpha <  1$ 
requires to compute the $1 - \alpha$ quantile of the $n 
\hLMMD^2$ distribution under $\hypot{0}$ denoted by $q_{\alpha, n}$. Thus $q_{\alpha, n}$ serves as a critical value to decide whether the test statistic $\hat{\Delta}^2$ is significantly close to $0$ or not, so that the probability of wrongly rejecting $\hypot{0}$ (Type-I error) is at most $\alpha$.   
\subsubsection{Classical parametric bootstrap}
In the case of a goodness-of-fit test, a usual way of estimating $q_{\alpha, n}$ is to perform a parametric bootstrap. Parametric bootstrap consists in generating $B$ samples of $n$ \iid random variables $Y_1^{(b)}, \ldots, Y_n^{(b)} \sim \mathcal{N}(\tilde{m}, \tilde{\Sigma})$ ($b = 1, \ldots, B$). Each of these $B$ samples is used to compute a bootstrap replication 
\begin{align}
[n \hLMMD^2]^{b} = n \Vert \hat{\bmu}_P^b - \NN{\tm^b}{\tsig^b} \Vert^2_{H(\bk)}  \ens ,
\label{slow.bt.rep}
\end{align}
where $\hat{\bmu}_P^b$, $\tm^b$ and $\tsig^b$ are the estimators of $\mu_P$, $m$ and $\Sigma$ based on $Y_1^b, \ldots, Y_n^b$.

It is known that parametric bootstrap is asymptotically valid \cite{Stute1993}. Namely, under $\hypot{0}$,
\begin{align*}
\forall b = 1, \ldots, B, \quad \left(n \hLMMD^2, \ens [n \hLMMD^2]^{b} \right) \underoverset{n \to +\infty}{\mathcal{L}}{\longrightarrow} (U, U^{'}) \ens ,
\end{align*}

where $U$ and $U^{'}$ are \iid random variables. In a  nutshell, \eqref{slow.bt.rep} is approximately an independent copy of the test statistic $n \hLMMD^2$ (under $\hypot{0}$). Therefore $B$ replications $[n \hLMMD^2]^{b}$ can be used to estimate the $1 - \alpha$ quantile $q_{\alpha, n}$ of $n \hLMMD^2$ under the null-hypothesis.

However, this approach suffers heavy computational costs. In particular, each bootstrap replication involves the estimators $(\tilde{m}^{b}, \tilde{\Sigma}^{b})$. In our case, this leads to compute the eigendecomposition of the $B$ Gram matrices $K^{b} = [\langle Y_i^{b}, Y_j^{b}\rangle ]_{i, j}$ of size $n \times n$ hence a complexity of order $\mathcal{O}(B n^3)$.
\subsubsection{Fast parametric bootstrap}
\label{sssec.fast.parambt}
This computational limitation is alleviated by means of another strategy described in \cite{Kojadinovic2012}. Let us consider in a first time the case when the estimators of $m$ and $\Sigma$ are the classical empirical mean and covariance $\hm = (1/n) \sum_{i=1}^n Y_i$ and $\hsig = (1/n) \sum_{i=1}^n (Y_i - \hm)^{\otimes 2}$. Introducing the Fr\'echet derivative \cite{Frigyik2008} $D_{(m, \Sigma)} N$ at $(m, \Sigma)$ of the function 
\begin{align*}
N : \Theta \to H(\bk), \ens (m, \Sigma) \mapsto \NN{m}{\Sigma} \ens ,
\end{align*}
our bootstrap method relies on the following approximation
\begin{align}
\sqrt{n}\left(\hat{\bmu}_P - \NN{\hm}{\hsig} \right) \simeq & \ens \sqrt{n}\left( \hat{\bmu}_P - \underbrace{\NN{m_0}{\Sigma_0}}_{=\bar{\mu}_P \text{ under } \mathcal{H}_0} - D_{(m_0, \Sigma_0)} \NN{\hat{m} - m_0}{\hat{\Sigma} - \Sigma_0 } \right) \notag \\
	\simeq & \ens \frac{1}{\sqrt{n}} \sum_{i=1}^n [\bar{k}(Y_i, .) - \bar{\mu}_P] \notag \\
	& \qquad - D_{(m_0, \Sigma_0)} \NN{Y_i - m_0}{ (Y_i - m_0)^{\otimes 2} - \Sigma_0 } \ens . \label{stattest.approx.bt}
\end{align}
Since \eqref{stattest.approx.bt} consists of a sum of centered independent terms (under $\hypot{0}$), it is possible to generate approximate independent copies of this sum via \textit{weighted} bootstrap \cite{Burke2000}. Given $Z_1^b, \ldots, Z_n^b$ \iid real random variables of mean zero and unit variance and $\bar{Z}^b$ their empirical mean, a bootstrap replication of \eqref{stattest.approx.bt} is given by
\begin{align}
\frac{1}{\sqrt{n}} \sum_{i=1}^n (Z_i^b - \bar{Z}^b) \left\{\bar{k}(Y_i, .)  - D_{(m_0, \Sigma_0)} \NN{Y_i}{(Y_i - m_0)^{\otimes 2} } \right\} \ens .
\label{fast.parambt1}
\end{align}
Taking the square of the norm of \eqref{fast.parambt1} in $H(\bk)$ and replacing the unknown true parameters $m_0$ and $\Sigma_0$ by their estimators $\hm$ and $\hsig$ yields the bootstrap replication $[n \hLMMD^2]^b_{fast}$ of $n \hLMMD^2$
\begin{align}
[n \hLMMD^2]^b_{fast} \overset{\Delta}{=} \left\Vert \sqrt{n} \left(\hat{\bmu}_P^b - D_{(\hat{m}, \hat{\Sigma})} \NN{\hat{m}^{b}}{\hat{\Sigma}^{b}} \right) \right\Vert ^2_{H(\bar{k})} \ens ,
\label{fast.parambt2}
\end{align}   
where 
\begin{align*}
\hat{\bmu}_P^b & = (1 /n) \sum_{i=1}^n (Z_i^b - \bar{Z}^b) \bk(Y_i, .) \ens , \\
 \hat{m}^{b} & = (1/n) \sum_{i=1}^n (Z_i^b - \bar{Z}^b) Y_i \ens , \\
 \hat{\Sigma}^{b} & = (1/n) \sum_{i=1}^n (Z_i^b - \bar{Z}^b) (Y_i - \hat{m}^b)^{\otimes 2}.
\end{align*}
%
%
%

Therefore this approach avoids the recomputation of parameters for each bootstrap replication, hence a computationnal cost of order $\mathcal{O}(B n^2)$ instead of $\mathcal{O}(B n^3)$. This is illustrated empirically in the right half of Figure \ref{fast.parambt.perf}.
\subsubsection{Fast parametric bootstrap for general parameter estimators}
\label{sssec.fast.parambt.gen}
The bootstrap method proposed by \cite{Kojadinovic2012} used in Section~\ref{sssec.fast.parambt} requires that the estimators $(\tm, \tsig)$ can be written as a sum of independent terms with an additive term which converges to $0$ in probability. Formally, $(\tm, \tsig) = (m_0, \Sigma_0) + (1/n) \sum_{i=1}^n \psi(Y_i) + \epsilon$ where $\Esp \psi(Y) = 0$, $\mathrm{Var}(\psi(Y)) < +\infty$ and $\epsilon \overset{\Prb}{\llim{n}{+\infty}} 0$. However there are some estimators which cannot be written in this form straightforwardly. This is the case for instance if we test whether data follow a Gaussian with covariance of fixed rank $r$ (as in Section~\ref{sec.covranksel}). In this example, the associated estimators are  $\tm = \hat{m} = (1/n) \sum_{i=1}^n Y_i$ (empirical mean) and $\tsig = \hsig_r = \sum_{s = 1}^r \hat{\lambda}_s \hat{\Psi}_s^{\otimes 2}$ where $(\hat{\lambda}_s)_s$ and $(\hat{\Psi}_s)_s$ are the eigenvalues and eigenvectors of the empirical covariance operator $\hat{\Sigma} = (1/n) \sum_{i=1}^n (Y_i - \hat{\mu})^{\otimes 2}$.

We extend \eqref{fast.parambt2} to the general case when $\Theta_0 \neq \Theta$ and the estimators $(\tm, \tsig)$ are not the classical $(\hm, \hsig)$. We assume that the estimators $(\tm, \tsig)$ are functions of the empirical estimators $\hm$ and $\hsig$, namely there exists a continuous mapping $\TT$ such that 
\begin{align*}
(\tm, \tsig) = \TT(\hm, \hsig), \text{ where } \TT(\Theta) \subseteq \Theta_0 \text{ and }  \TT|_{\Theta_0} = \mathrm{Id}_{\Theta_0}.
\end{align*}
Under this definition, $(\tm, \tsig)$ are consistent estimators of $(m, \Sigma)$ when $(m, \Sigma) \in \Theta_0$. This kind of estimators are met for various choices of the null-hypothesis:
\begin{itemize}
\item \textbf{Unknown mean and covariance:} $(\tm, \tsig) = (\hm, \hsig)$ and $\TT$ is the identity map $\mathrm{Id}_{\Theta}$,
\item \textbf{Known mean and covariance:} $(\tm, \tsig) = (m_0, \Sigma_0)$ and $\TT$ is the constant map $\TT(m, \Sigma) = (m_0, \Sigma_0)$, 
\item \textbf{Known mean and unknown covariance:} $(\tm, \tsig) = (m_0, \hsig)$ and $\TT(m, \Sigma) = (m_0, \Sigma)$,
\item \textbf{Unknown mean and covariance of known rank $r$:} $(\tm, \tsig) = (\hm, \hsig_r)$ and $\TT(m, \Sigma) = (m, \Sigma_r)$ where $\Sigma_r$ is the rank $r$ truncation of $\Sigma$.
\end{itemize}
By introducing $\TT$, we get a similar approximation as in \eqref{stattest.approx.bt} by replacing the mapping $N: \Theta_0 \to H(\bk)$ with $N o \TT: \Theta_0 \to H(\bk)$. This leads to the bootstrap replication
\begin{align}
[n \hLMMD^2]^b_{fast} := \left\Vert \sqrt{n} \left(\hat{\bmu}_P^b - D_{(\hat{m}, \hat{\Sigma})} \NNoT{\hat{m}^{b}}{\hat{\Sigma}^{b}} \right) \right\Vert ^2_{H(\bar{k})} \ens .
\label{fast.parambt3}
\end{align} 
The validity of this bootstrap method is justified in Section \ref{ssec.theoassess.t1e}. 

Finally we define an estimator $\hat{q}_{\alpha, n}$ of $q_{\alpha, n}$ from the generated $B$ bootstrap replications $[n \hLMMD^2]^{1}_{fast} < \ldots <  [n \hLMMD^2]^{B}_{fast}$ (assuming they are sorted) 
\begin{align*}
\hat{q}_{\alpha, n} = [n \hLMMD^2]^{(\lfloor (1 - \alpha) B \rfloor)} \ens ,
\end{align*}
where $\lfloor . \rfloor$ stands for the integer part.
The rejection region is defined by
\begin{align*}
\mathcal{R}_{\alpha} = \{ n \hLMMD^2 > \hat{q}_{\alpha, n} \} \ens . 
\end{align*}

\subsubsection{Validity of the fast parametric bootstrap}
\label{ssec.theoassess.t1e}
Proposition \ref{prop.altern.parambt} hereafter shows the validity of the fast parametric bootstrap as presented in Section \ref{sssec.fast.parambt.gen}. The proof of Proposition \ref{prop.altern.parambt} is provided in Section \ref{proof.prop.altern.parambt}.

\begin{proposition}
\label{prop.altern.parambt}
Assume $\Esp_P \bk^{1/2}(Y, Y), \ens \mathrm{Tr}(\Sigma)$ and $\Esp_P ||Y - m_0||^4$ are finite. Also assume that $\TT$ is continuously differentiable on $\Theta_0$.

If $\hypot{0}$ is true, then for each $b = 1, \ldots, B$,
\begin{enumerate}
\item[(i)] $\sqrt{n} \left(\hat{\bmu}_P - \NN{\tilde{m}}{\tilde{\Sigma}}\right) \underoverset{n \to +\infty}{\mathcal{L}}{\longrightarrow} G_P - D_{(m_0, \Sigma_0)} \NNoTbis{U_P} $
\item[(ii)] $\sqrt{n} \left(\hat{\bmu}_P^b - {D}_{(\hat{m}, \hat{\Sigma})} \NNoT{\hat{m}^{b}}{ \hat{\Sigma}^{b}} \right) \underoverset{n \to +\infty}{\mathcal{L}}{\longrightarrow}  G^{'}_P - D_{(m_0, \Sigma_0)} \NNoTbis{U^{'}_P}$
\end{enumerate}
where $(G_P, U_P)$ and $(G^{'}_P, U^{'}_P)$ are $i.i.d.$ random variables in $H(\bar{k}) \times \Theta$. \\
If otherwise $\hypot{0}$ is false, $(ii)$ is still true. \\
Furthermore, $G_P$ and $U_P$ are zero-mean Gaussian r.v. with covariances 
%
\begin{align*}
\mathrm{Var}\left(G_P \right) & = \Esp_{Y \sim P} (\bar{k}(Y, .) - \bar{\mu}_P)^{\otimes 2}\\
\mathrm{Var}\left( U_P \right)  & = \Esp_{Y \sim P} \left[Y - m_0, (Y - m_0)^{\otimes 2} - \Sigma \right]^{\otimes 2}\\
\mathrm{cov}\left( G_P, U_P \right) & = \Esp_{Y \sim P} (\bar{k}(Y, .) - \bar{\mu}_P) \otimes \left[Y - m_0, (Y - m_0)^{\otimes} - \Sigma_0 \right]  \ens .
\end{align*}
\end{proposition}
By the Continous Mapping Theorem and the continuity of $\Vert . \Vert_{H(\bk)}^2$, Proposition \ref{prop.altern.parambt} guarantees that the estimated quantile converges almost surely to the true one as $n, B \to +\infty$, so that the type-I error equals $\alpha$ asymptotically. 

Note that in \cite{Kojadinovic2012} the parameter subspace $\Theta_0$ must be a subset of $\R^p$ for some integer $p \geq 1$. Proposition~\ref{prop.altern.parambt} allows $\Theta_0$ to be a subset of a possibly infinite-dimensional Hilbert space ($m$ belongs to $\hsp$ and $\Sigma$ belongs to the space of finite trace operators $\hsp \to \hsp$). 

Figure~\ref{fast.parambt.perf} (left plot) compares empirically the bootstrap distribution of $[n \hLMMD^2]^b_{fast}$ and the distribution of $n \hLMMD^2$. When $n = 1000$, the two corresponding densities are superimposed and a two-sample Kolmogorov-Smirnov test returns a p-value of $0.978$ which confirms the strong similarity between the two distributions. Therefore the fast bootstrap method seems to provide a very good approximation of the distribution of $n \hLMMD^2$ even for a moderate sample size $n$.
\begin{figure}[t]
\vskip 0.2in
\begin{center}
\centerline{
\includegraphics[width=0.8\textwidth]{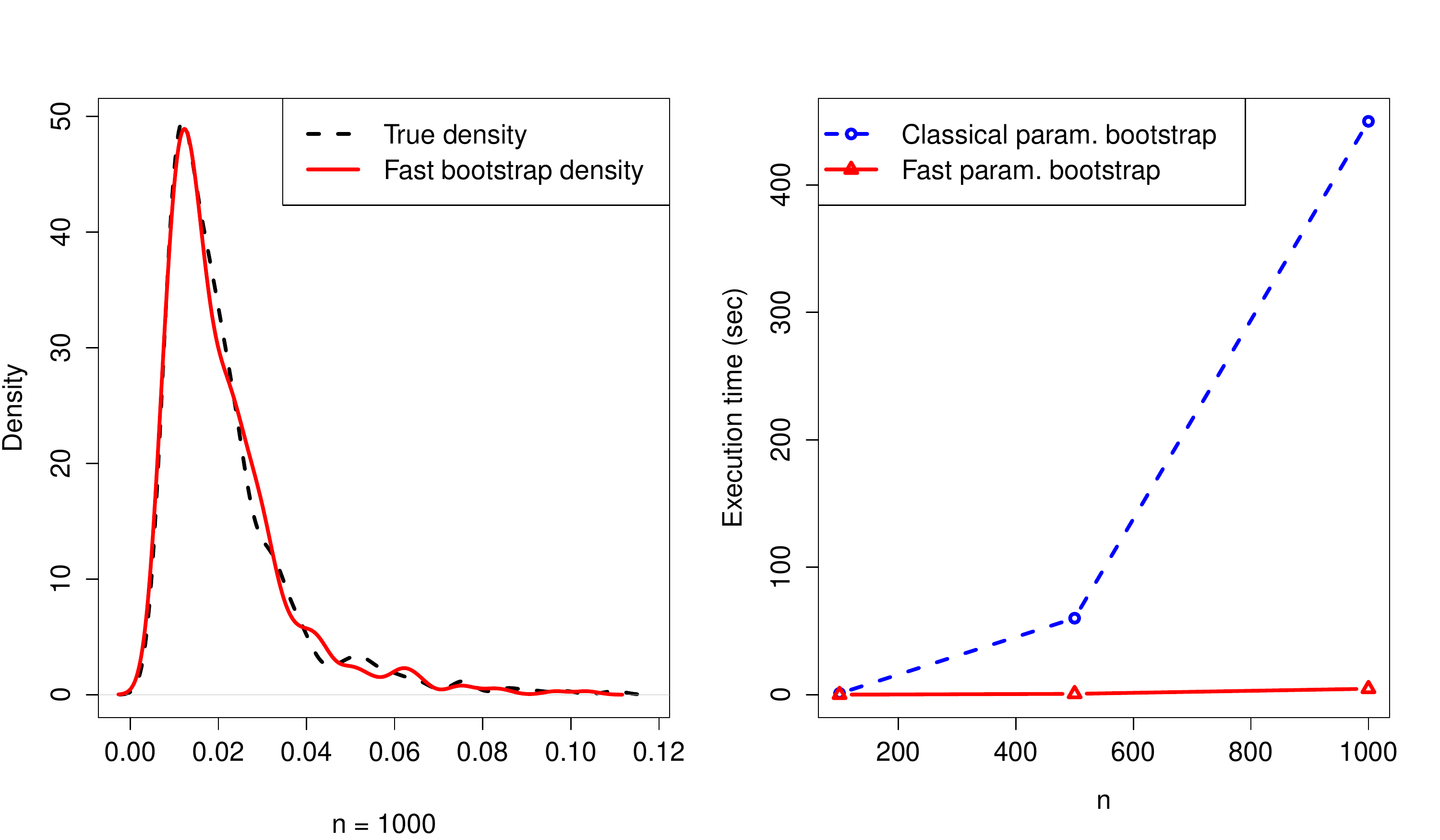}
}
\caption{\textbf{Left:} Comparison of the distributions of $n \hLMMD^2$ (test statistic) and $[n \hLMMD^2]^b_{fast}$ (fast bootstrap replication) when $n = 1000$. A Kolmogorov-Smirnov two-sample test applied to our simulations returns a p-value of $0.978$ which confirms the apparent similarity between the two distributions. \textbf{Right: } Comparison of the execution time (in seconds) of both classical and fast bootstrap methods.} 
\label{fast.parambt.perf}
\end{center}
\vskip -0.2in
\end{figure}
\section{Test performances}
\label{sec.theoretical.assessment}

\subsection{An upper bound for the Type-II error}
\label{ssec.theoassess.t2e}
Let us assume the null-hypothesis is false, that is $P \neq \mathcal{N}(m_0, \Sigma_0)$ or $(m_0, \Sigma_0) \notin \Theta_0$. Theorem~\ref{PropTypeIIerr} gives the magnitude of the Type-II error, that is the probability of wrongly accepting $\hypot{0}$. The proof can be found in Appendix \ref{appendix.proof.prop.t2e}.

Before stating Theorem~\ref{PropTypeIIerr}, let us introduce or recall useful notation :
\begin{itemize}
\item $\Delta = \left\Vert \bar{\mu}_{P} - \NNoT{m_0}{ \Sigma_0}] \right\Vert_{H(\bar{k})} $,
\item $q_{\alpha, n} = \Esp \hat{q}_{\alpha, n} $ ,
\item $m^{2}_P = \Esp_P ||D_{(m_0, \Sigma_0)} \NNoTbis{\Psi(Y)} - \bk(Y, .) + \bar{\mu}_P ||_{H(\bar{k})}^2$,
\end{itemize}
where $\Psi(Y) = (Y - m_0, [Y - m_0]^{\otimes 2} - \Sigma_0)$ and $D_{(m_0, \Sigma_0)} (N o \TT)$ denotes the Fr\'echet derivative of $N o \TT$ at $(m_0, \Sigma_0)$. According to Proposition \ref{prop.altern.parambt} and the continuous mapping theorem, $\hat{q}_{\alpha, n}$ corresponds to an order statistic of a random variable which converges weakly to $\left\Vert G^{'}_P - D_{(m_0, \Sigma_0)} \NNoTbis{U^{'}_P}  \right\Vert^2$ (as defined in Proposition \ref{prop.altern.parambt}). Therefore, its mean $q_{\alpha, n}$ tends to a finite quantity as $n \to +\infty$. $L$ and $m_P^2$ do not depend on $n$ as well.
%
%

%
%
%
\begin{theorem}{(Type II error)}
\label{PropTypeIIerr}
Assume $\sup_{x, y \in \hsp_0} |\bar{k}(x, y)| = M < + \infty$ where $Y \in \hsp_0 \subseteq \hsp$ $P$-almost surely and $\hat{q}_{\alpha, n}$ is independent of $n \hLMMD^2$. \\
Then, for any $n > q_{\alpha, n} \Delta^{-2}$
\begin{align}
\label{TypeIIerrBound}
\Prb\left( n \hLMMD^2 \leq \hat{q}_{\alpha, n} \right) \leq \exp\left( - \frac{n \left(\Delta - \displaystyle\frac{q_{\alpha, n}}{n \Delta} \right)^2}{2 m_P^2 + C m_P M^{1/2}(\Delta^2 - q_{\alpha, n}/n) }    \right) f(\alpha, B, M, \Delta) \ens ,
\end{align}
where
\begin{align}
f(\alpha, B, M, \Delta) & = (1+o_n(1)) \left(1 + \frac{C_{P^b}}{C^{'} \Delta^2 M^{1/2} m_P \sqrt{\alpha B}} + \frac{o_B(B^{-1/2})}{C^{''} \Delta^4 M m_P^2} \right) \notag \ens ,
\end{align}
and $C, C^{'}, C^{''}$ are absolute constants and $C_{P^b}$ only depends on the distribution of $[n \hLMMD^2]^b_{fast}$.
\end{theorem} 
The first implication of Proposition \eqref{PropTypeIIerr} is that our test is consistent, that is
\begin{align*}
\Prb ( n \hLMMD^2 \leq \hat{q}_{\alpha, n} \mid \hypot{0} \text{ false}) \underset{n \to +\infty}{\longrightarrow} 0 \ens .
\end{align*}  

Furthermore, the upper bound in \eqref{TypeIIerrBound} reflects the expected behaviour of the Type-II error with respect to meaningful quantities. When $\Delta$ decreases, the bound increases (alternative more difficult to detect). When $\alpha$ (Type-I error) decreases, $q_{\alpha, n}$ gets larger and $n$ has to be larger to get the bound. The variance term $m_P^2$ encompasses the difficulty of estimating $\bar{\mu}_P$ and of estimating the parameters as well. In the special case when $m$ and $\Sigma$ are known, $\TT = Id$ and the chain rule yields $D_{(m_0, \Sigma_0)} (N o \TT) = (D_{\TT(m_0, \Sigma_0)} N )o (D_{(m_0, \Sigma_0)} \TT) = 0$ so that $m_P^2 = \Esp ||\bar{\phi}(Y) - \bar{\mu}_P||^2$ reduces to the variance of $\hat{\bmu}_P$. As expected, a large $m_P^2$ makes the bound larger.
Note that the estimation of the critical value which is related to the term $f(\alpha, B, M, \Delta)$ in \eqref{TypeIIerrBound} does not alter the asymptotic rate of convergence of the bound.

Remark that assuming that $\hat{q}_{\alpha, n}$ is independent of $n \hLMMD^2$ is reasonable for a large $n$, since $n \hLMMD^2$ and $\hat{q}_{\alpha, n}$ are asymptotically independent according to Proposition \ref{prop.altern.parambt}.

\label{sec.experimental.results}

\subsection{Empirical study of type-I/II errors}\label{ssc.t12e.study}

\begin{figure}
\vskip 0.2in
\begin{center}
\centerline{
\includegraphics[width=0.75\textwidth]{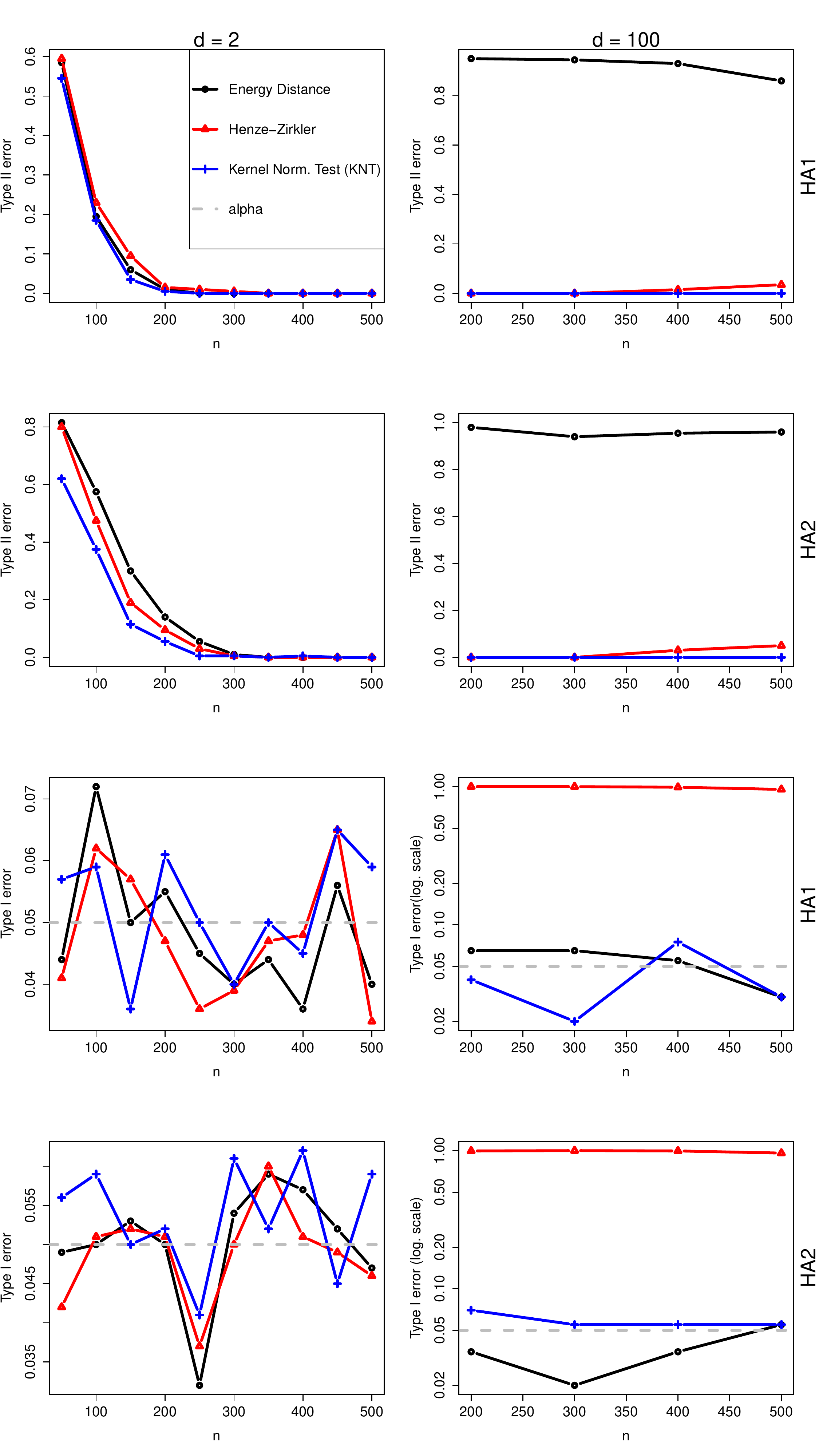}
}
\caption{Type-I and type-II errors of KNT ($+$ blue), Energy Distance ($\circ$ black), and Henze-Zirkler ($\bigtriangleup$ red). Two alternative distributions are considered: HA1 (rows $1$ and $3$) and HA2 (rows $2$ and $4$). Two settings are considered: $d=2$ (left) and $d = 100$ (right).}
\label{T2E_FinDim}
\end{center}
\vskip -0.2in
\end{figure}

Empirical performances of our test are inferred on the basis of synthetic data. For the sake of brevity, our test is referred to as KNT (Kernel Normality Test) in the following.

One main concern of goodness-of-fit tests is their drastic loss of power as dimensionality increases. 
Empirical evidences (see Table 3 in \cite{SzekelyRizzoENormTest}) prove ongoing multivariate normality tests suffer such deficiencies.
The purpose of the present section is to check if KNT displays a good behavior in high or infinite dimension.

%

\subsubsection{Finite-dimensional case (Synthetic data)}
\label{subsubsec.finite.dimension}
\textbf{Reference tests.} The power of our test is compared with that of two multivariate normality tests: the Henze-Zirkler test (HZ) \cite{HenzeZirkler} and the energy distance (ED) test  \cite{SzekelyRizzoENormTest}. 
The main idea of these tests is briefly recalled in Appendix \ref{appendix.mvn.tests.hz} and \ref{appendix.mvn.tests.ed}.

\textbf{Null and alternative distributions.} 
Two alternatives are considered: a mixture of two Gaussians with different means ($\mu_1 = 0$ and $\mu_2 = 1.5 \enspace (1, 1/2, \ldots, 1/d)$) and same covariance $\Sigma = 0.5 \enspace \mathrm{diag}(1, 1/4, \ldots, 1/d^2)$, whose mixture proportions equals either $(0.5, 0.5)$ (alternative HA1) or $(0.8, 0.2)$ (alternative HA2). Furthermore, two different cases for $d$ are considered: $d=2$ (small dimension) and $d=100$ (large dimension).

\textbf{Simulation design.} $200$ simulations are performed for each test, each alternative and each $n$ (ranging from $100$ to $500$). $B$ is set at $B = 250$ for KNT. The test level is set at $\alpha = 0.05$ for all tests. 

\textbf{Results.} 
%
In the small dimension case (Figure \ref{T2E_FinDim}, left column), the actual Type-I error of all tests remain more or less around $\alpha$ ($\pm 0.02$). Their Type-II errors are superimposed and quickly decrease down to $0$ when $n \geq 200$.
On the other hand, experimental results reveal different behaviors as $d$ increases (Figure \ref{T2E_FinDim}, right column). 
Whereas ED test lose power, KNT and HZ still exhibits small Type-II error values. 
Besides, ED and KNT Type-I errors remain around the prescribed level $\alpha$ while that of HZ is close to $1$, which shows that its small Type-II error is artificial.
This seems to confirm that HZ and ED tests are not suited to high-dimensional settings unlike KNT.
\subsubsection{Infinite-dimensional case (real data)}
\label{RealDataPowerStudy}
\begin{figure}
\vskip 0.2in
\begin{center}
\centerline{
\includegraphics[width=0.65\textheight]{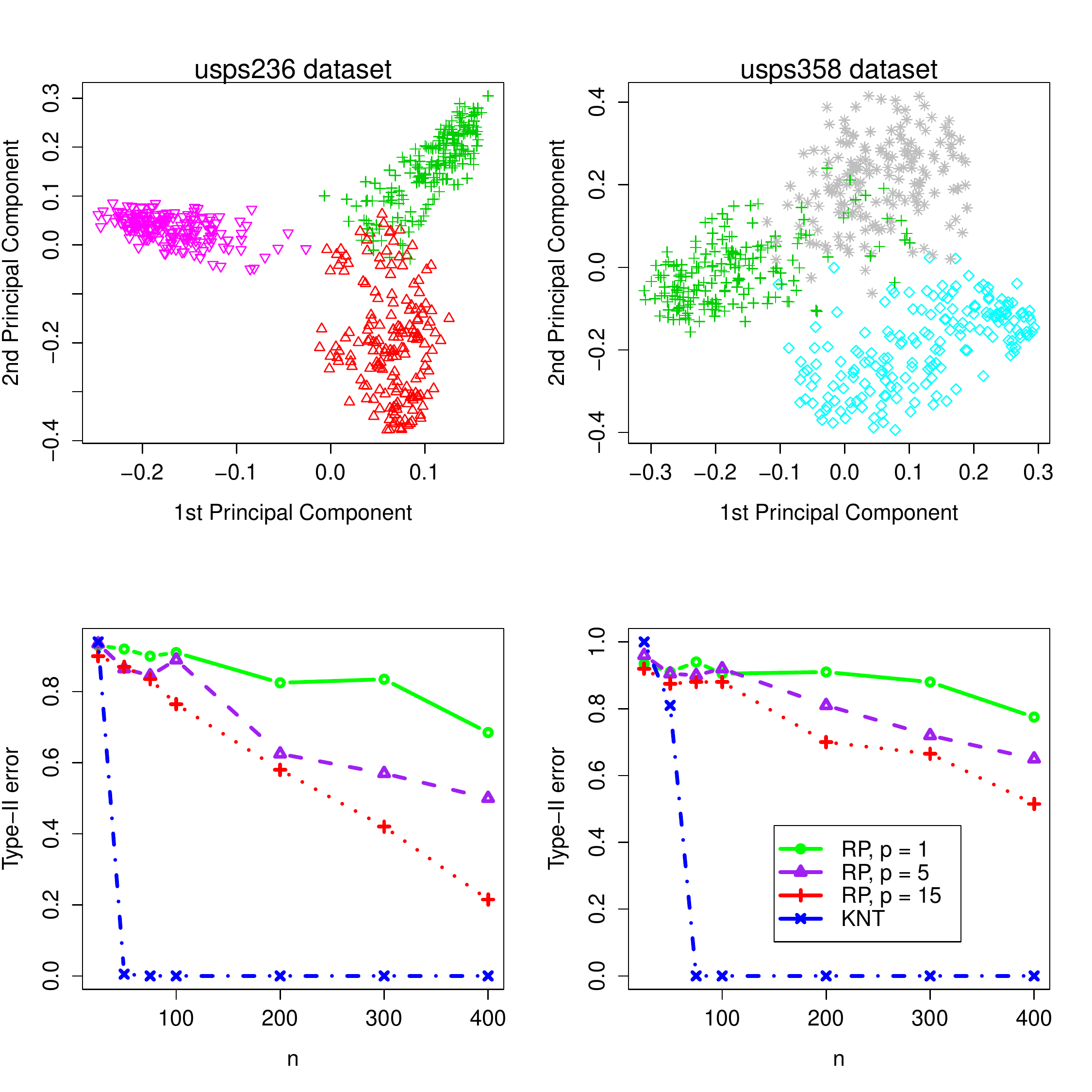}
}
\caption{3D-Visualization (Kernel PCA) of the "Usps236" (top row, left) and "Usps358" (top row, right) datasets; comparison of Type-II error (bottom row, left: "Usps236", right: "Usps358") for: KNT ($\times$ blue) and Random Projection with $p = 1$ ($\bullet$ green), $p = 5$ ($\bigtriangleup$ purple) and $p = 15$ ($+$ red) random projections. \label{RealD}}
\end{center}
\vskip -0.2in
\end{figure}
\textbf{Dataset and chosen kernel.} Let us consider the USPS dataset which consists of handwritten digits represented by a vectorized $8 \times 8$ greyscale matrix ($\mathcal{X} = \mathbb{R}^{64}$).
A Gaussian kernel $k_G(\cdot,\cdot) = \exp(-\sigma^2 ||\cdot\,-\,\cdot ||^2)$ is used with $\sigma^2 = 10^{-4}$.
%
%
%
Comparing sub-datasets "Usps236" (keeping the three classes "$2$", "$3$" and "$6$", $541$ observations) and "Usps358" (classes "3", "5" and "8", $539$ observations), the 3D-visualization (Figure \ref{RealD}, top panels) suggests
three well-separated Gaussian components for ``Usps236'' (left panel), and more overlapping classes for ``Usps358'' (right panel).

\textbf{References tests.} KNT is compared with Random Projection (RP) test, specially designed for infinite-dimensional settings. RP is presented in Appendix \ref{appendix.mvn.tests.rp}. Several numbers of projections $p$ are considered for the RP test : $p = 1, 5 $ and $15$.

\textbf{Simulation design.} We set $\alpha = 0.05$ and $200$ repetitions have been done for each sample size.  

\textbf{Results.} (Figure~\ref{RealD}, bottom plots) 
RP is by far less powerful KNT in both cases, no matter how many random projections $p$ are considered. Indeed, KNT exhibits a Type-II error near $0$ when $n$ is barely equal to $100$, whereas RP still has a relatively large Type-II error when $n = 400$. 
On the other hand, RP becomes more powerful as $p$ gets larger as expected. A large enough number of random projections may allow RP to catch up KNT in terms of power. But RP has a computational advantage over KNT only when $p = 1$ where the RP test statistic is distribution-free. This is no longer the case when $p \geq 2$ and the critical value for the RP test is only available through Monte-Carlo methods.  
%
%

%
%
%

%
%
\section{Application to covariance rank selection}
\label{sec.covranksel}
\subsection{Covariance rank selection through sequential testing}
\label{ssec.covrankestim}
Under the Gaussian assumption, the null hypothesis becomes 
\begin{align*}
\hypot{0}: (m_0, \Sigma_0) \in \Theta_0 \ens ,
\end{align*}
and our test reduces to a test on parameters.

We focus on the estimation of the rank of the covariance operator $\Sigma$. Namely, we consider a collection of models $(\mathcal{M}_r)_{1 \leq r \leq r_{max}}$ such that, for each $r = 1, \ldots, r_{max}$,
\begin{align*}
\mathcal{M}_r = \left\{ P = \mathcal{N}(m, \Sigma_r) \mid   m \in H(k) \text{ and } \mathrm{rk}(\Sigma_r) = r \right\} \ens .
\end{align*} 
Each of these models correspond respectively to the following null hypotheses
\begin{align*}
H_{0, r} : \mathrm{rank}(\Sigma) = r, \ens r = 1, \ldots, r_{max} \ens ,
\end{align*} 
and the corresponding tests can be used to select the most reliable model. These tests are performed in a sequential procedure summarized in Algorithm~\ref{algo.seqtest}.
This sequential procedure yields an estimator $\hat{r}$ defined as
\begin{align*}
\hat{r} \overset{\Delta}{=} \underset{\tilde{r}}{\mathrm{min}} \left\{ H_{0, r} \text{ rejected for } r = 1, \ldots, \tilde{r}-1 \text{ and } H_{0, \tilde{r}} \text{ accepted}  \right\} \ens .
\end{align*}
or $\hat{r} \overset{\Delta}{=} r_{max}$ if all of the hypotheses are rejected.

Sequential testing to estimate the rank of a covariance matrix (or more generally a noisy matrix) is mentionned in \cite{Ratsimalahelo_2003} and \cite{RobinSmith_2000}. Both of these papers focus on the probability to select a wrong rank, that is $\mathbb{P}(\hat{r} \neq r^*)$ where $r^*$ denotes the true rank. The goal is to choose a level of confidence $\alpha$ such that this probability of error converges almost surely to $0$ when $n \to +\infty$. 

There are two ways of guessing a wrong rank : either by overestimation or by underestimation. Getting $\hat{r}$ greater than $r^*$  implies that the null-hypothesis $H_{0, r^*}$ was tested and wrongly rejected, hence a probability of overestimating $r^*$ at most equal to $\alpha$. Underestimating means that at least one of the false null-hypothesis $H_{0, 1}, \ldots, H_{0, r^* - 1}$ was wrongly accepted (Type-II error). Let $\beta_r(\alpha)$ denote the Type-II error of testing $H_{0, r}$ with confidence level $\alpha$ for each $r < r^*$. Thus by a union bound argument, 
\begin{align}
\mathbb{P}(\hat{r} \neq r^*) \leq \sum_{r = 1}^{r^* - 1} \beta_r(\alpha) + \alpha \ens .
\label{bound.prob.wrongrk}
\end{align}
The bound in \eqref{bound.prob.wrongrk} decreases to $0$ only if $\alpha$ converges to $0$ but at a slow rate. Indeed, the Type-II errors $\beta_r(\alpha)$ grow with decreasing $\alpha$ but converge to zero when $n \to +\infty$. For instance in the case of the sequential tests mentionned in \cite{Ratsimalahelo_2003} and \cite{RobinSmith_2000}, the correct rate of decrease for $\alpha$ must satisfy $(1/n) \log(1/\alpha) = o_n(1)$.  
%
%
%
%
%
\begin{algorithm}[t]
\caption{Sequential selection of covariance rank}
\label{algo.seqtest}
\begin{algorithmic}
\normalsize
\Require{Gram matrix $K = [\bk(Y_i, Y_j)]_{i,j}$, confidence level $0 < \alpha < 1$}
\begin{enumerate}
\item[1.] Set $r = 1$ and test $H_{0, r}$
\item[2.] If $H_{0, r}$ is rejected and $r < r_{max}$, set $r = r+1$ and return to $1.$
\item[3.] Otherwise, set the estimator of the rank $\hat{r} = r$.
\end{enumerate}
\Return{estimated rank $\hat{r}$ }
\end{algorithmic}
\end{algorithm}
%
%
%
%
\begin{figure}[t]
\vskip 0.2in
\begin{center}
\centerline{
\includegraphics[width=0.6 \columnwidth]{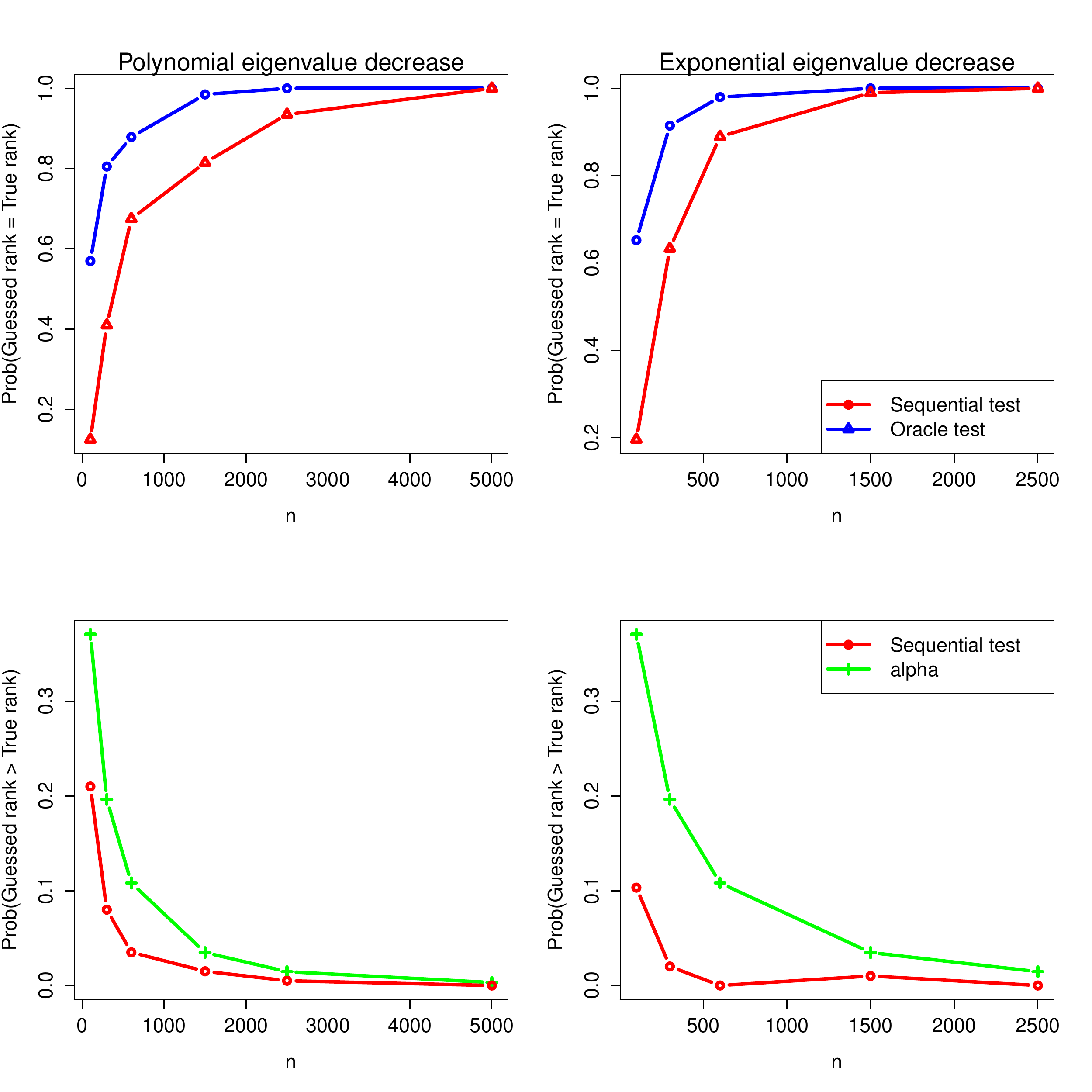}
}
\caption{\textbf{Top half:} Probabilities of finding the right rank with respect to $n$ for our sequential test ($\bullet$ red) and the oracle procedure ($\bigtriangleup$ blue); \textbf{bottom half:} probabilities of overestimating the true rank with the sequential procedure compared with fixed alpha ($+$ green). In each case, two decreasing rate for covariance eigenvalues are considered : polynomial (left column) and exponential (right column).  
\label{seqtest.probs}}
\end{center}
\vskip -0.2in
\end{figure} 
\subsection{Empirical performances}
\label{ssc.covranksel} 
In this section, the sequential procedure to select covariance rank (as presented in Section \ref{ssec.covrankestim}) is tested empirically on synthetic data.

\textbf{Dataset} A sample of $n$ zero-mean Gaussian with covariance $\Sigma_{r^*}$ are generated, where $n$ ranges from $100$ to $5000$. $\Sigma_{r^*}$ is of rank $r^{*} = 10$ and its eigenvalues decrease either polynomially ($\lambda_r = r^{-1}$ for all $r \leq r^{*}$) or exponentially ($\lambda_r = \exp(-0.2 r)$ for all $r \leq r^{*}$).

\textbf{Benchmark} To illustrate the level of difficulty, we compare our procedure with an oracle procedure which uses the knowledge of the true rank. Namely, the oracle procedure follows our sequential procedure at a level $\alpha_{oracle}$  defined as follows
\begin{align*}
\alpha_{oracle} = \underset{1 \leq r \leq r^{*} - 1}{\mathrm{max}}  \Prb_Z(n \hLMMD^2_r \leq Z_r) \ens ,
\end{align*}
where $n \hLMMD^2_r$ is the observed statistic for the $r$-th test and $Z_r$ follows the distribution of this statistic under $H_{0, r}$.
Hence $\alpha_{oracle}$ is chosen such that the true rank $r^{*}$ is selected whenever it is possible. 

\textbf{Simulation design} To get a consistent estimation of $r^{*}$, the confidence level $\alpha$ must decrease with $n$ and is set at $\alpha = \alpha_n = \exp(-0.125 n^{0.45})$. Each time, $200$ simulations are performed.

\textbf{Results} The top panels of Figure \ref{seqtest.probs} display the proportion of cases when the target rank is found, either for our sequential procedure or the oracle one. When the eigenvalues decay polynomially, the oracle shows that the target rank cannot be almost surely guessed until $n = 1500$. When $n \leq 1500$, our procedure finds the true rank with probability at least $0.8$ and quickly catches up to the oracle as $n$ grows. In the exponential decay case, a similar observation is made. This case seems to be easier, as our procedure performs almost as well as the oracle when $n \geq 600$. In all cases, the consistency of our procedure is confirmed by the simulations.

The bottom panels of Figure \ref{seqtest.probs} compare $\alpha$ with the probability of overestimating $r^{*}$ (denoted by $p_+$). As noticed in Section \ref{ssec.covrankestim}, the former is an upper bound of the latter. But we must check empirically whether the gap between those two quantities is not too large, otherwise the sequential procedure would be too conservative and lead to excessive underestimation of $r^{*}$. In the polynomial decay case, the difference between $\alpha$ and $p_+$ is small, even when $n=100$. The gap is larger in the exponential case but gets broader when $n \leq 1500$.

\subsection{Robustness analysis}
In practice, none of the models $\mathcal{M}_r$ is true. An additive full-rank noise term is often considered in the literature \citep{Choi2014, Josse2012}. Namely, we set in our case
\begin{align}
Y = Z + \epsilon
\label{noisy.model}
\end{align} 
where $Z \sim \mathcal{N}(m, \Sigma_{r^{*}})$ with $\mathrm{rk}(\Sigma_{r^{*}}) = r^{*}$ and $\epsilon$ is the error term independent of $Z$. Note that the Gaussian assumption concerns the main signal $Z$ and not the error term whereas usual models assume the converse \citep{Choi2014, Josse2012}. 

Figure \ref{seqtest.robustness} illustrates the performance of our sequential procedure under the noisy model \eqref{noisy.model}. We set $\hsp = \Reals^{100}$, $n= 600$, $r^{*} = 3$ and $\Sigma_{r^{*}} = \Sigma_3 = \mathrm{diag}(\lambda_1, \ldots, \lambda_3, 0, \ldots, 0)$ where $\lambda_r = \exp( - 0.2 r)$ for $r \leq 3$. The noise term is $\epsilon = (\lambda_{3} \rho^{-1} \eta_i)_{1 \leq i \leq 100}$ where $\eta_1, \ldots, \eta_{100}$ are \iid Student random variables with $10$ degrees of freedom and $\rho > 0$ is the \textit{signal-to-noise ratio}. 

As expected, the probability of guessing the target rank $r^{*}$ decreases down to $0$ as the signal-to-noise ratio $\rho$ diminishes. However, choosing a smaller level of confidence $\alpha$ allows to improve the probability of right guesses for a fixed $\rho$. without sacrificing much for smaller signal-to-noise ratios. This is due to the fact that each null-hypothesis $H_{0, r}$ is false, hence the need for a smaller $\alpha$ (smaller Type-I error) which yields greater Type-II errors and avoids the rejection of all of the null-hypotheses. 
\begin{figure}[t]
\vskip 0.2in
\begin{center}
\centerline{
\includegraphics[width=0.55 \columnwidth]{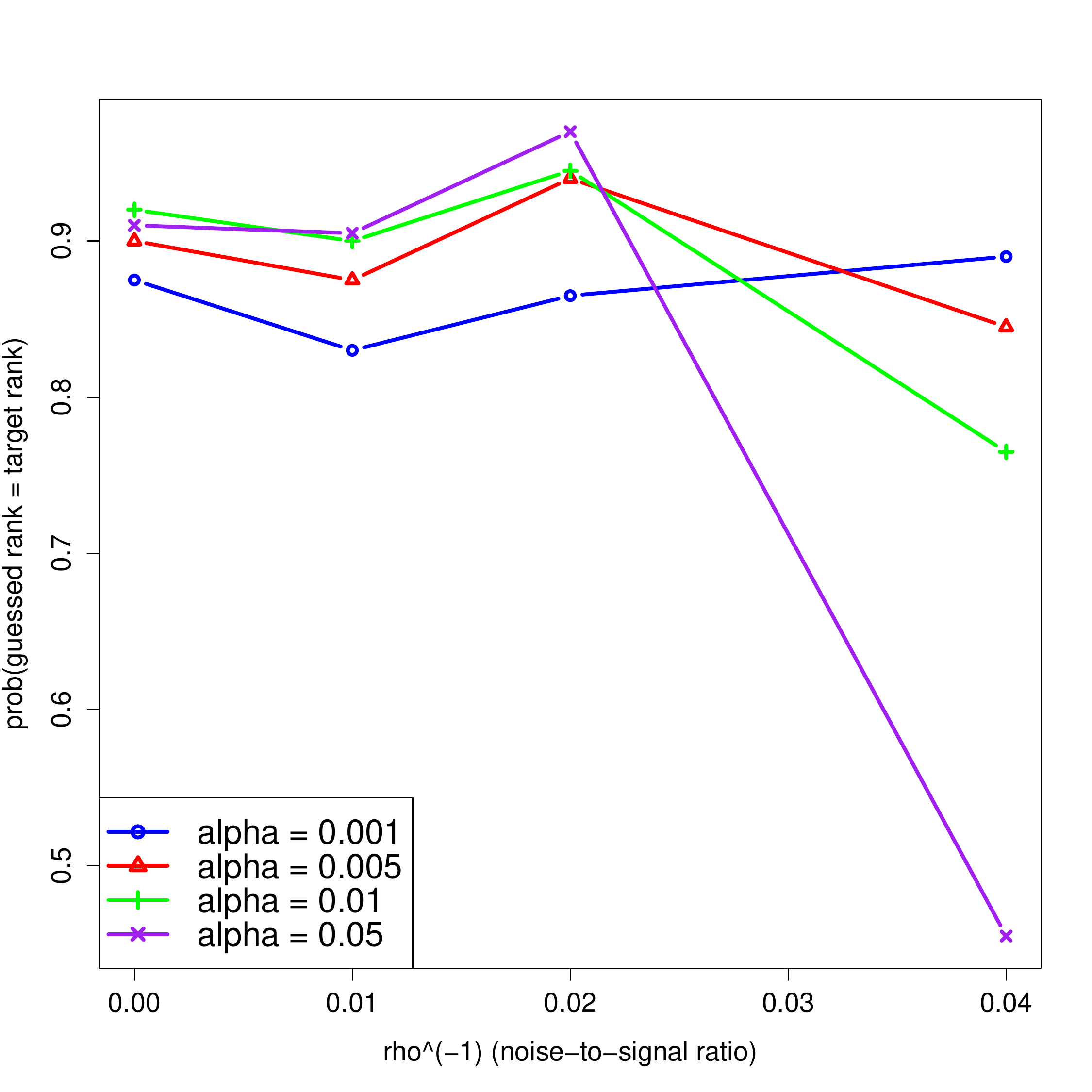}
}
\caption{Illustration of the robustness of our sequential procedure under a noisy model. } 
\label{seqtest.robustness}
\end{center}
\vskip -0.2in
\end{figure} 
%
%
\FloatBarrier
\section{Conclusion}
We introduced a new normality test suited to high-dimensional Hilbert spaces. 
It turns out to be more powerful than ongoing high- or infinite-dimensional tests (such as random projection). 
In particular, empirical studies showed a mild sensibility to high-dimensionality. Therefore our test can be used as a multivariate normality (MVN) test without suffering a loss of power when $d$ gets larger unlike other MVN tests (Henze-Zirkler, Energy-distance). 

If the Gaussian assumption is validated beforehand, our test becomes a general test on parameters. It is illustrated with an application to covariance rank selection that plugs our test into a sequential procedure. Empirical evidences show the good performances and the robustness of this method.    

As for future improvements, investigating the influence of the kernel $\bar{k}$ on the performance of the test would be of interest. In the case of the Gaussian kernel for instance, a method to optimize the Type-II error with respect to the hyperparameter $\sigma$ would be welcomed. This aspect has just began to be studied in \cite{Gretton2012_OptKern} when performing homogeneity testing with a convex combination of kernels.

Finally, the choice of the level $\alpha$ for the sequential procedure (covariance rank selection) is another subject for future research. Indeed, an asymptotic regime for $\alpha$ has been exhibited to get consistency, but setting the value of $\alpha$ when $n$ is fixed remains an open question.


%
%

\bibliographystyle{plain}
\bibliography{Refs}


\appendix 

\section{Goodness-of-fit tests}

\subsection{Henze-Zirkler test}
\label{appendix.mvn.tests.hz}
The Henze-Zirkler test \cite{HenzeZirkler} relies on the following statistic
\begin{equation}
HZ = \int_{\Reals^d} \abs{ \hat{\Psi}(t) - \Psi(t) }^2 \omega(t) dt \ens ,
\end{equation}
where $\Psi(t)$ denotes the characteristic function of $\mathcal{N}(0,I)$, $\hat{\Psi}(t) = n^{-1} \sum_{j = 1}^n e^{i \langle t, Y_j\rangle }$ is the empirical characteristic function of the sample $Y_1, \dots, Y_n$, and $\omega(t) = (2 \pi \beta)^{-d/2} \exp(-||t||^2/(2 \beta))$ with $\beta = 2^{-1/2} [(2d + 1) n)/4]^{1/(d+4)}$. The $\hypot{0}$-hypothesis is rejected for large values of $HZ$. Note that the sample $Y_1,\ldots,Y_n$ must be whitened (centered and renormalized) beforehand.

\subsection{Energy distance test}
\label{appendix.mvn.tests.ed}
The energy distance test \cite{SzekelyRizzoENormTest} is based on 
\begin{equation}
\mathcal{E}(P, P_0) = 2 \Esp ||Y - Z|| - \Esp ||Y - Y^{\prime}|| - \Esp ||Z - Z^{\prime}|| \ens 
\end{equation}
which is called the \emph{energy distance}, where $Y, Y^{\prime} \sim P$ and $Z, Z^{\prime} \sim P_0$. 
Note that $\mathcal{E}(P, P_0) = 0$ if and only if $P = P_0$.
The test statistic is given by
\begin{align}
\hat{\mathcal{E}} = & \frac{2}{n} \sum_{i=1}^n \Esp_Z || Y_i - Z || - \Esp_{Z, Z^{\prime}} ||Z - Z^{\prime}|| \notag \\
         & \hspace*{2cm} -\frac{1}{n^2} \sum_{i, j = 1}^n ||Y_i - Y_j|| \hfill \ens ,
\end{align}
where $Z, Z^{\prime} \stackrel{\iid}{\sim} \mathcal{N}(0, I)$ (null-distribution).
HZ and ED tests set the $\hypot{0}$-distribution at $P_0 = \mathcal{N}(\hat{\mu}, \hat{\Sigma})$ where $\hat{\mu}$ and $\hat{\Sigma}$ are respectively the standard empirical mean and covariance.
As for the Henze-Zirkler test, data must be centered and renormalized.

\subsection{Projection-based statistical tests}
\label{appendix.mvn.tests.rp}
In the high-dimensional setting, several approaches share a common idea consisting in projecting on one-dimensional spaces. This idea relies on the Cramer-Wold theorem extended to infinite dimensional Hilbert space.
\begin{proposition}{(Prop.~2.1 from \cite{RandomProjTest})} 
Let $\mathcal{H}$ be a separable Hilbert space with inner product $\langle \cdot , \cdot\rangle $, and $Y,Z\in\mathcal{H}$ denote two random variables with respective Borel probability measures $P_Y$ and $P_Z$. 
If for every $h \in \mathcal{H}$, $\langle Y, h\rangle  = \langle Z, h\rangle $ weakly then $ P_Y = P_Z $.
\end{proposition}
%
\cite{RandomProjTest} suggest to randomly choose directions $h$ from a Gaussian measure and perform a Kolmogorov-Smirnov test on $\prodscal{Y_1}{h} , \ldots, \prodscal{Y_n}{h}$ for each $h$, leading to the test statistic 
\begin{align}
D_n(h) = \sup_{x \in \Reals} |\hat{F}_n (x) - F_0(x)|
\end{align}
where $\displaystyle \hat{F}_n(x) = (1/n) \sum_{i=1}^n \indc_{\{\langle Y_i, h\rangle \leq x \}}$ is the empirical cumulative distribution function (cdf) of $\langle Y_1, h\rangle ,\ldots,\langle Y_n,h\rangle $ and $F_0(x) = \Prb( \langle Y, h\rangle \leq x )$ denotes the cdf of $\langle Z, h\rangle $.

Since \cite{RandomProjTest} proved too few directions lead to a less powerful test, this can be repeated for several randomly chosen directions $h$, keeping then the largest value for $D_n(h)$.
However the test statistic is no longer distribution-free (unlike the univariate Kolmogorov-Smirnov one) when the number of directions is larger than 2. 

\section{Proofs}

Throughout this section, $\langle .,.\rangle $ (resp. $||.||$) denotes either $\langle .,.\rangle _\hsp$ or $\langle ~.~,.\rangle _{H(\bk)}$ (resp. $||.||_\hsp$ or $||.||_{H(\bk)}$) depending on the context.

\subsection{Proof of Propositions \ref{prop.teststat.gauss} and \ref{prop.teststat.exp}}
\label{appendix.proof.prop.teststat}
%
%
Consider the eigenexpansion of $\tsig = \sum_{i \geq 1} \lambda_i \Psi_i^{\otimes 2}$ where $\lambda_1, \lambda_2, \ldots$ is a decreasing sequence of positive reals and where $\{\Psi_{i}\}_{i \geq 1}$ form a complete orthonormal basis of $\hsp$. 

Let $Z \sim \mathcal{N}(\tm, \tsig)$. The orthogonal projections $\langle Z, \Psi_i\rangle $ are $\mathcal{N}(0, \lambda_{i})$ and for $i \neq j$, $\langle Z, \Psi_i\rangle $ and $\langle Z, \Psi_{j}\rangle $ are independent. Let $Z^{'}$ be an independent copy of $Z$.

\bigskip
$\bullet$ \textbf{Gaussian kernel case : } $\bk(.,.) = \exp( - \sigma \Vert \cdot - \cdot \Vert^2_\hsp)$

Let us first expand the following quantity
\begin{align}
\NN{\tm}{\tsig}(y)  & =  \Esp_{Z} \exp\left( - \sigma ||Z - y||^2 \right) \notag \\
                     & =  \Esp_{Z} \exp\left( - \sigma \sum_{i \geq 1} \langle Z - y, \Psi_i \rangle ^2 \right) \notag \\
                     & = \prod_{i=1}^{+ \infty} \Esp_{Z} \exp\left(- \sigma \sum_{i \geq 1} \langle Z - y, \Psi_i \rangle ^2 \right)  \label{expandnorm0.gauss} \\
                     & = \prod_{i=1}^{+ \infty} \Esp_Z \exp\left(-\sigma \langle Z - y, \Psi_i\rangle^2 \right)   \notag\\
                     & = \prod_{i=1}^{+ \infty}  (1+2\sigma \lambda_r)^{-1/2} \exp\left( - \sigma \frac{\langle  \tm - y, \Psi_r\rangle ^2}{1 + 2 \sigma \lambda_r}   \right) \notag \\
                     & = \left| I + 2 \sigma \tsig \right|^{-1/2} \exp(-\sigma \langle (I + 2 \sigma \tsig)(y - \tm), y - \tm \rangle ) \notag
\end{align}
We can switch the mean and the limit in \eqref{expandnorm0.gauss}   by using the Dominated Convergence theorem since for every $N \geq 1$
\begin{equation*}
\left| \prod_{i=1}^{N} \exp\left(- \sigma \langle Z - y, \Psi_i \rangle ^2 \right) \right| \leq 1 < +\infty \ens .
\end{equation*}
%
%
%
The second quantity $||\NN{\tm}{\tsig}||^{2}$ is computed likewise
\begin{align}
||\NN{\tm}{\tsig}||^{2} & =  \Esp_{Z, Z^{'}} \exp\left( - \sigma ||Z - Z^{'}||^2 \right) \notag \\
                     & =  \Esp_{Z, Z^{'}} \exp\left( - \sigma \sum_{i \geq 1} \langle Z - Z^{'}, \Psi_i \rangle ^2 \right) \notag \\
                     & = \prod_{i=1}^{+ \infty} \Esp_{Z, Z^{'}} \exp\left(- \sigma \sum_{i \geq 1} \langle Z - Z^{'}, \Psi_i \rangle ^2 \right)  \notag \\
                     & = \prod_{i=1}^{+ \infty} \Esp_Z \Esp_{Z^{'}}\left(\exp\left(-\sigma \langle Z -Z^{'}, \Psi_i\rangle \right) \mid Z\right)  \notag\\
                     & = \prod_{i=1}^{+ \infty} \Esp_Z (1+2\sigma \lambda_r)^{-1/2} \exp\left( - \sigma \frac{\langle  \tm - Z, \Psi_r\rangle ^2}{1 + 2 \sigma \lambda_r}   \right) \notag \\
                     & = \prod_{i=1}^{+ \infty} (1+2\sigma \lambda_r)^{-1/2} \Esp_{U \sim \mathcal{N}(0, \lambda_r)} \exp\left( - \frac{\sigma U^2}{1 + 2 \sigma \lambda_r}   \right) \notag \\
                     & = \prod_{i=1}^{+ \infty} (1+2\sigma \lambda_r)^{-1/2} \left( 1 + \frac{2 \sigma \lambda_r}{1 + 2 \sigma \lambda_r} \right)^{-1/2} \notag \\
                     & = \prod_{i=1}^{+ \infty} (1 + 4 \sigma \lambda_r)^{-1/2} \notag \\
                     & = \left| I + 4 \sigma \tsig\right|^{-1/2} \notag
\end{align}

\bigskip
$\bullet$ \textbf{Exponential kernel case : } $\bk(.,.) = \exp(\langle \cdot,  \cdot\rangle _\hsp)$

Let us first expand the following quantity
\begin{align}
\NN{\tm}{\tsig}(y)  & =  \Esp_{Z} \exp\left( \langle Z, y\rangle  \right) \notag \\
                     & =  \Esp_{U \sim \mathcal{N}(\langle \tm, y\rangle , \langle \tsig y, y\rangle )} \exp\left( U \right) \notag \\
                     & = \exp(\langle \tm, y\rangle  + (1/2) \langle \tsig y, y\rangle ) \notag .
\end{align}

Expanding $||\NN{\tm}{\tsig}||^{2}$,
\begin{align}
||\NN{\tm}{\tsig}||^{2} & =  \Esp_{Z, Z^{'}} \exp\left(\langle Z, Z^{'}\rangle \right) \notag \\
                     & =  \Esp_{Z, Z^{'}} \exp\left(\sum_{i \geq 1} \langle Z, \Psi_i\rangle  \langle Z^{'}, \Psi_i\rangle \right) \notag \\
                     & = \prod_{i=1}^{\infty} \Esp_{Z, Z^{'}} \exp\left(\langle Z, \Psi_i\rangle  \langle Z^{'}, \Psi_i\rangle \right) \ens .
                     \label{expandnorm.exp}
\end{align}
We can switch the mean and the limit by using the Dominated Convergence theorem since
\begin{equation*}
\left|\prod_{i=1}^{N} \exp\left(\langle Z, \Psi_i\rangle  \langle Z^{'}, \Psi_i\rangle \right)\right| \overset{a.s.}{\underset{N \to \infty}{\longrightarrow}} \exp\left(\langle Z, Z^{'}\rangle \right) \leq \exp\left(\frac{||Z||^2 + ||Z^{'}||^2}{2}\right) \ens ,
\end{equation*}
and 
\begin{equation*}
\Esp_{Z, Z^{'}} \exp\left(\frac{||Z||^2 + ||Z^{'}||^2}{2}\right) = \left[\Esp_Z \exp\left(\frac{||Z||^2}{2}\right)\right]^{2} = \Esp \bk^{1/2}(Z, Z) < +\infty \ens .
\end{equation*}
The integrability of $\bk^{1/2}(Z, Z)$ is necessary to ensure the existence of the embedding $\NN{\tm}{\tsig}$ related to the distribution of $Z$. As we will see thereafter, it is guaranteed by the condition $\hat{\lambda}_1 < 1$.

For each $i$,
\begin{align}
\Esp_{Z, Z^{'}} \exp\left(\langle Z, \Psi_i\rangle  \langle Z^{'}, \Psi_i\rangle \right) & = \Esp_Z \Esp_{Z^{'}}\left(\exp\left(\langle Z, \Psi_i\rangle  \langle Z^{'}, \Psi_i\rangle \right) \mid Z\right)\notag\\
                     & = \Esp_Z \exp\left(\frc{\lambda_i}{2} \langle Z, \Psi_i\rangle ^2\right)\notag\\
					 & = (1 - \lambda^2_i)^{-1/2} \label{expandEsp.exp}
\end{align}

Plugging \eqref{expandEsp.exp} into Equation \eqref{expandnorm.exp},
                     
\begin{align}                     
 ||\NN{\tm}{\tsig}||^{2} = \prod_{i=1}^{\infty} (1 - \gamma_{i}^{2})^{-1/2} = \left|I - \Sigma^2 \right|^{-1/2} \ens .
                     \label{FinalProofEq2}
\end{align}

\eqref{FinalProofEq2} is well defined only if $\left|I - \Sigma^2 \right| > 0$. As we have assumed that $ \lambda_i < 1$ for all $i$, it is positive. The non-nullity also holds since
\begin{align}
\prod_{i=1}^{\infty} [1-\lambda_{i}^{2}] = \exp\left(- \sum_{i=1}^{\infty} \log\left( \frac{1}{1 - \lambda_{i}^{2}} \right) \right) & \geq  \exp\left( - \sum_{i=1}^{\infty} \left( \frac{1}{1 - \lambda_{i}^{2}} - 1\right)  \right)  \notag\\
  & =  \exp\left(- \sum_{i=1}^{\infty} \frac{\lambda_{i}^{2}}{1 - \lambda_{i}^{2}} \right)\notag\\
  & \geq  \exp\left( \frac{- \mathrm{Tr}( \Sigma^2 )}{1 - \lambda_1^2 }  \right)  \ens  . \label{checkStrictPosA}
\end{align}
where we used the inequality: $\log(x) \leq x - 1$. Since $\Sigma$ is a finite trace operator, its eigenvalues converge towards $0$ and $\lambda_i^2 \leq \lambda_i < 1$ for $i$ large enough. Thus, $\mathrm{Tr} (\Sigma^2)$ is finite and it follows from \eqref{checkStrictPosA} that $\left| I - \Sigma^2 \right| > 0$.

%

\subsection{Proof of Theorem \ref{prop.altern.parambt}}
\label{proof.prop.altern.parambt}

The proof of Proposition \ref{prop.altern.parambt} follows the same idea as the original paper \cite{Kojadinovic2012}  and broadens its field of applicability. Namely, the parameter space does not need to be a subset of $\Reals^d$ anymore. The main ingredient for our version of the proof is to use the CLT in Banach spaces \cite{HoffmannJorgensen1976} instead of the multiplier CLT for empirical processes (\cite{Kosorok2007}, Theorem $10.1$).   

We introduce the following notation : 
\begin{itemize}
\item $\theta_0 = (m_0, \Sigma_0)$ denotes the true parameters
\item $\theta_n = (\hat{m}, \hat{\Sigma})$ denotes the empirical parameters
\item For any $\theta = (m, \Sigma)$ and $y \in \hsp$, $\Psi_\theta(y) = (y - m, (y - m)^{\otimes 2} - \Sigma)$.
\item $\theta_{0,n}^b = (1/n) \sum_{i=1}^n (Z^b_i - \bar{Z}^b)\Psi_{\theta_0} (Y_i)$ and $\theta_n^b = (1/n) \sum_{i=1}^n (Z^b_i - \bar{Z}^b)\Psi_{\theta_n} (Y_i)$
%
\end{itemize}

Define the covariance operators $C_1: H(\bar{k}) \to H(\bar{k})$, $C_2: \Theta \to \Theta$ and $C_{1,2}: \Theta \to H(\bar{k})$ as
\begin{align*}
C_1  & = \mathrm{Var}\left(\sqrt{n} (\hat{\bmu}_P - \bar{\mu}_P) \right) = \Esp_Y (\bar{k}(Y, .) - \bar{\mu}_P)^{\otimes 2}\\
C_2 & = \mathrm{Var}\left( n^{-1/2} \sum_{i=1}^n \Psi_{\theta_0} (Y_i) \right) = \Esp_Y \Psi_{\theta_0}(Y)^{\otimes 2}\\
C_{1,2} & = \mathrm{cov}\left( \sqrt{n} (\hat{\bmu}_P - \bar{\mu}_P) , n^{-1/2} \sum_{i=1}^n \Psi_{\theta_0} (Y_i) \right) = \Esp_Y (\bar{k}(Y, .) - \bar{\mu}_P) \otimes \Psi_{\theta_0}(Y)  \ens .
\end{align*}
From \cite{HoffmannJorgensen1976}, the CLT is verified in a Hilbert space under the assumption of finite second moment (satisfied in our case since $\mathrm{Tr}(C_1) = \Esp_P \bar{k}(Y, Y) - ||\bar{\mu}_P||^2 < +\infty$ and $\mathrm{Tr}(C_2) = \Esp_P ||Y - \mu_0||^4 + \mathrm{Tr}(\Sigma_0 - \Sigma_0^2) < +\infty$ by assumption). Therefore,
\begin{align*}
\sqrt{n} (\hat{\bmu}_P - \bar{\mu}_P) & \underoverset{n \to +\infty}{\mathcal{L}}{\longrightarrow} G_P \sim \mathcal{GP}(0, C_1) \\
n^{-1/2} \sum_{i=1}^n \Psi_{\theta_0} (Y_i) & \underoverset{n \to +\infty}{\mathcal{L}}{\longrightarrow} U_P \sim \mathcal{GP}(0, C_2)
\end{align*}
Introducing
\begin{align*}
C_1^b := \mathrm{Var}\left(\sqrt{n} \bar{\mu}^b_{\hat{P}}  \right)&  = n^{-1} \sum_{i=1}^n \Esp(Z_i - \bar{Z})^2 C_1 = (1 - 1/n) C_1 \\
C_2^b := \mathrm{Var}\left(\sqrt{n} \theta^b_{0,n} \right)&  = n^{-1} \sum_{i=1}^n \Esp(Z_i - \bar{Z})^2 C_2 = (1 - 1/n) C_2 \\
C_{1,2}^b := \mathrm{cov}\left(\sqrt{n} \theta^b_{0,n} , \sqrt{n} \hat{\bmu}^b_{P} \right) & = n^{-1} \sum_{i=1}^n \Esp(Z_i - \bar{Z})^2 C_{1, 2} = (1 - 1/n) C_{1,2} \ens ,
\end{align*}
we derive likewise
\begin{align*}
\sqrt{n} \hat{\bmu}^b_{P} & \underoverset{n \to +\infty}{\mathcal{L}}{\longrightarrow} G^{'}_P \sim \mathcal{GP}(0, C_1)\\
\sqrt{n} \theta^b_{0, n} & \underoverset{n \to +\infty}{\mathcal{L}}{\longrightarrow}  U^{'}_P \sim \mathcal{GP}(0, C_2) \ens .
\end{align*}
Since
\begin{equation*}
\begin{array}{ll}
\mathrm{cov}\left(\sqrt{n} \hat{\bmu}^b_{P} , \sqrt{n} (\hat{\bmu}_P - \bar{\mu}_P) \right)  = 0 , &   \mathrm{cov}\left(\sqrt{n} \theta^b_{0,n} , n^{-1/2} \sum_{i=1}^n \Psi_{\theta_0} (Y_i) \right)  =  0 \\
&\\
\mathrm{cov}\left(\sqrt{n} \hat{\bmu}^b_{P} , n^{-1/2} \sum_{i=1}^n \Psi_{\theta_0} (Y_i) \right) = 0 , & \mathrm{cov}\left(\sqrt{n} \theta^b_{0,n} , \sqrt{n} (\hat{\bmu}_P - \bar{\mu}_P) \right)  =  0 \ens ,
\end{array}
\end{equation*}
the limit Gaussian processes $(G_P, U_P)$ and $(G^{'}_P, U^{'}_P)$ are independent.\\
Since $D_{\theta_0} (N o \TT)$ is assumed continuous w.r.t. $\theta_0$, we get by the continuous mapping theorem that
\begin{align*}
\left(\sqrt{n} (\hat{\bmu}_P - \bar{\mu}_P) - D_{\theta_0} \NNoT{n^{-1/2} \sum_{i=1}^n \Psi_{\theta_0}(Y_i)]}{\sqrt{n} \bar{\mu}^b_{\hat{P}} - D_{\theta_0} (N o \TT) [\sqrt{n} \theta_{0,n}^b]} \right) \ens , 
\end{align*}
converges weakly to
\begin{align*}
\left(G_P - D_{\theta_0} \NNoTbis{U_P}, G^{'}_P - D_{\theta_0} \NNoTbis{U^{'}_P}\right) \ens .
\end{align*}
To get the final conclusion of Proposition \ref{prop.altern.parambt}, we have to prove two remaining things.\\
First, under the assumption that $P = \mathcal{N}(\TT(\theta_0))$,
\begin{align*}
\sqrt{n} (\hat{\bmu}_P - N o \TT[\theta_n])  = & \sqrt{n} (\hat{\bmu}_P - \bar{\mu}_{P}) + \sqrt{n} (N o \TT[\theta_0] - N o \TT[\theta_n]) \\
 = & \sqrt{n} (\hat{\bmu}_P - \bar{\mu}_{P}) - D_{\theta_0} (N o \TT)[\sqrt{n} (\theta_n - \theta_0)] + o_P(\sqrt{n} || \theta_n - \theta_0 ||) \\
 = & \sqrt{n} (\hat{\bmu}_P - \bar{\mu}_{P}) - D_{\theta_0} (N o \TT)[n^{-1/2} \sum_{i=1}^n \Psi_{\theta_0}(Y_i)] \\
 & + D_{\theta_0} (N o \TT)[o_P(1)] + o_P(\sqrt{n} || \theta_n - \theta_0 ||) \\
 \underoverset{n \to +\infty}{\mathcal{L}}{\longrightarrow} & G_P - D_{\theta_0} (N o \TT)[U_P]  \ens , 
\end{align*}
because $\sqrt{n}(\theta_n - \theta_0)$ converges weakly to a zero-mean Gaussian and by using the continous mapping theorem with the continuity of $\theta_0 \mapsto D_{\theta_0}\NN$ and of $||.||_{\Theta}$. \\
Secondly, whether $\hypot{0}$ is true or not,
\begin{align*}
\sqrt{n} \hat{\bmu}^b_{P} - D_{\theta_n} (N o \TT) [\sqrt{n} \theta_{n}^b] = &  \sqrt{n} \hat{\bmu}^b_{P} - D_{\theta_0} (N o \TT) [\sqrt{n} \theta_{0,n}^b]  + \underbrace{D_{\theta_n} (N o \TT) [\sqrt{n}(\theta_{0,n}^b - \theta_{n}^b)]}_{:=(a)} \\
& +\underbrace{ D_{\theta_0} (N o \TT) [\sqrt{n} \theta_{0,n}^b] - D_{\theta_n} (N o \TT) [\sqrt{n} \theta_{0,n}^b] }_{:=(b)}  \ens .  
\end{align*}
we must check that both $(a)$ and $(b)$ converge $P$-almost surely to $0$.\\
Since by the continuous mapping theorem
\begin{align*}
(a) =  \underbrace{\left(n^{-1/2} \sum_{i=1}^n (Z_i - \bar{Z})\right)}_{\longrightarrow \mathcal{N}(0, 1)}  D_{\theta_n} (N o \TT)[ \underbrace{(m_0 - m, \Sigma_0 - \Sigma_n}_{\to 0 \text{ a.s.}})] \underoverset{n \to +\infty}{P - a.s.}{\longrightarrow} 0 \ens ,
\end{align*}
and since 
\begin{align*}
(b) = D_{\theta_0} (N o \TT) [\sqrt{n} \theta_{0,n}^b] - D_{\theta_n} (N o \TT) [\sqrt{n} \theta_{0,n}^b] \longrightarrow 0 \ens P\text{-almost surely} \ens ,
\end{align*}
as $\theta_n(\omega) \to \theta_0$ $P$-almost surely and $N$ is continuous w.r.t. $\theta$, it follows
\begin{align*}
\sqrt{n} \bar{\mu}^b_{\hat{P}} - D_{\theta_n} (N o \TT) [\sqrt{n} \theta_{n}^b] \longrightarrow G^{'}_P - D_{\theta_0} (N o \TT)[U^{'}_P] \ens ,
\end{align*}
hence the conclusion of Proposition \ref{prop.altern.parambt}.
%

\subsection{Proof of Theorem \ref{PropTypeIIerr}}
\label{appendix.proof.prop.t2e}


The goal is to get an upper bound for the Type-II error 
\begin{align}
\Prb(n \hDn^2 \leq \hat{q} \mid \hypot{1}) \ens . \label{typeIIerror} 
\end{align}
In the following, the feature map from $H(k)$ to $H(\bar{k})$ will be denoted as
\begin{align*}
\bar{\phi} : H(k) \to H(\bar{k}), \ens y \mapsto \bar{k}(y, .) \ens .
\end{align*}
Besides, we use the shortened notation $q := \Esp \hat{q}_{\alpha, n, B}$ for the sake of simplicity (see Section~\ref{ssec.theoassess.t2e} for definitions).
\begin{enumerate}
\item \textbf{Reduce $n \hDn^2$ to a sum of independent terms}

The first step consists in getting a tight upper bound for \eqref{typeIIerror} which involve a sum of independent terms. This will allow the use of a Bennett concentration inequality in the next step.\\
Introducing the Fr\'echet derivative $D_{\theta_0} (N o \TT)$ of $N o \TT$ at $\theta_0$, $n \hDn^2$ is expanded as follows
\begin{align}
n \hDn^2 = &  \frac{1}{n} \sum_{i, j = 1}^n \langle  \bar{\phi}(Y_i) - N o \TT (\theta_n),  \bar{\phi}(Y_j) - N o \TT (\theta_n)\rangle  \notag \\
		:= & n \hDn^2_{0} + n \Delta^2 + 2 n S_n + \epsilon \ens .
		\label{hoeffding.exp}
\end{align}
where 
\begin{align*}
\hDn_0^2 =  \frac{1}{n^{2}} \sum_{i, j = 1}^n \langle  \Xi(Y_i),  \Xi(Y_j)\rangle  \ens ,
\end{align*}
\begin{align*}
S_n = \frac{1}{n} \sum_{i=1}^n \langle  \bar{\mu}_P - N o \TT(\theta_0), \Xi(Y_i)\rangle  \ens ,
\end{align*}
\begin{align*}
\Xi(Y_i) = D_{\theta_0} (N o \TT)(\Psi(Y_i)) - \bar{\phi}(Y_i) + \bar{\mu}_P \ens ,
\end{align*}
and $\epsilon = o_n(1)$ almost surely.

$n \hDn^2_{0}$ is a degenerate U-statistics so it converges weakly to a sum of weighted chi-squares (\cite{SerflingBook}, page 194). $\sqrt{n} S_n$ converges weakly to a zero-mean Gaussian by the classic CLT as long as $\Esp \langle  \bar{\mu}_P - N o \TT(\theta_0), \Xi(Y_i)\rangle ^2$ is finite (which is true since $\bar{k}$ is bounded). It follows that $\hDn^2_{0}$ becomes negligible with respect to $S_n$ when $n$ is large. Therefore, we consider a surrogate for the Type-II error \eqref{typeIIerror} by removing $\hDn^2_{0}$ with a negligible loss of accuracy.
Plugging \eqref{hoeffding.exp}  into \eqref{typeIIerror} 
\begin{align}
\Prb(n \hDn^2 \leq \hat{q}) = & \Prb(n \hDn^2_{0} + n L^2 + 2 n S_n  \leq \hat{q} - \epsilon ) \notag \\
	= & (1 + o_n(1)) \ens \Prb(n \hDn^2_{0} + n L^2 + 2 n S_n  \leq \hat{q}  ) \ens .
\end{align}
%
%
Finally, using $\hDn_0^2 \geq 0$ and conditionning on $\hat{q}$ yield the upper bound
\begin{align}
\Prb(n \hDn^2 \leq \hat{q} \mid \hat{q}) \ens \leq  (1 + o_n(1)) \ens \Prb\left(\sum_{i=1}^n f(Y_i) \geq n \hat{s} \right) \label{typeIIerr.first.bound} \ens ,
\end{align}
where 
\begin{align}
f(Y_i) := \ens  \langle  \bar{\mu}_P - N o \TT(\theta_0), \Xi(Y_i)\rangle  \ens , \quad 
\hat{s}	:= \ens L^2 - \frac{\hat{q}}{n}  \ens . \notag
\end{align}
%
%
%
\item \textbf{Apply a concentration inequality}

We now want to find an upper bound for \eqref{typeIIerr.first.bound} through a concentration inequality, namely Lemma~\ref{LemmaBennett} with $\xi_i = f(Y_i)$, $\epsilon = n \hat{s}$, $\nu^2 = \mathrm{Var}(f(Y_i))$ and $f(Y_i) \leq c = \bar{M}$ ($P$-almost surely). 

Lemma~\ref{LemmaBennett} combined with Lemma~\ref{LemmaVariance} and \ref{LemmaBound} 
yields the upper bound
\begin{align}
\Prb(\sum_{i=1}^n f(Y_i) \geq n \hat{s} \mid  \hat{q}) \leq & \exp\left( - \frac{n \hat{s}^2}{2 \vartheta^2 + (2/3) \overline{M} \vartheta \hat{s}}  \right) \indc_{\hat{s} \geq 0} +  \indc_{\hat{s} < 0} \notag \\
	:= & \exp(g(\hat{s})) \indc_{\hat{s} \geq 0} +  \indc_{\hat{s} < 0} := h(\hat{s}) \label{BennettBound} \ens ,
\end{align}
where
\begin{align}
\overline{M} := \left(4 + \sqrt{2}\right) L M^{1/2} \ens , \qquad \nu^2 \leq \vartheta^2 := \ens L^2 m_P^{2} \ens , \notag
\end{align}
and $m_P^{2} = \Esp \left\Vert \Xi(Y) \right\Vert^2$. \\
\item \textbf{"Replace" the estimator $\hat{q}_{\alpha, n}$ with the true quantile $q_{\alpha, n}$ in the bound}

It remains to take the expectation with respect to $\hat{q}$. In order to make it easy, $\hat{q}$ is pull out of the exponential term of the bound. This is done through a Taylor-Lagrange expansion (Lemma~\ref{LemmaTaylorLag}).

Lemma~\ref{LemmaTaylorLag} rewrites the bound in \eqref{BennettBound} as
\begin{align}
\exp\left( - \frac{n s^2}{2 \vartheta^2 + (2/3) \overline{M} \vartheta  s}  \right) \left\{1 + \frac{3 n }{2 \overline{M} \vartheta} \exp\left( \frac{3 |\tilde{q} - q|}{2 \overline{M} \vartheta }  \right) \indc_{\tilde{s} \geq 0} |\hat{s} - s| \right\} \ens , \label{TaylorLag}
\end{align}
where 
\begin{align}
s = L^2 - \frc{q}{n}  \ens , \quad \tilde{s} = L^2 - \frc{\tilde{q}}{n}  \ens , \quad \tilde{q} \in (q \wedge \hat{q}, q \vee \hat{q}) \ens \notag,
\end{align}
and $s \geq 0$ because of the assumption $n > q L^{-2}$.

The mean (with respect to $\hat{q}$) of the right-side multiplicative term of \eqref{TaylorLag} is bounded by
\begin{align}
1 + \frac{3 n }{2 \overline{M} \vartheta } \left\{\Esp_{\hat{q}}\left(\exp\left( \frac{3 |\tilde{q} - q|}{\overline{M} \vartheta }  \right) \indc_{\tilde{s} \geq 0} \right)\right\}^{1/2} \sqrt{\Esp_{\hat{q}} (\hat{s} - s)^2} \ens , \notag 
\end{align}

because of the Cauchy-Schwarz inequality.\\
On one hand, $\Esp (\hat{q} - q)^2 \underset{B \to +\infty}{\to} 0$ (Lemma~\ref{LemmaGapQtl}) implies $\Esp (\tilde{q} - q)^2 \underset{B \to +\infty}{\to} 0$ and thus $\tilde{q} \underset{B \to +\infty}{\to} q_\infty$ weakly where $q_\infty = \lim_{n \to +\infty} q_{\alpha, n}$ (that is almost surely for $q_\infty$ is a constant). Hence
\begin{align}
\Esp_{\hat{q}}\left(\exp\left( \frac{3 |\tilde{q} - q|}{\overline{M} \vartheta }  \right) \indc_{\tilde{s} \geq 0} \right) = & \ens \Esp_{\hat{q}} \left(\left[1 + \frac{o_B(|\hat{q} - q|)}{\overline{M} \vartheta} \right] \indc_{\tilde{s} \geq 0} \right) \notag \\
	\leq & \ens 1 + \frac{\Esp_{\hat{q}} (o_B(|\hat{q} - q|) \indc_{\tilde{s} \geq 0})}{\overline{M} \vartheta} \ens , \notag
\end{align}
where $o_B$ denotes almost sure convergence.\\
Since the variable $|\tilde{q} - q| \indc_{\tilde{s} \geq 0}$ is bounded by the constant $|n L^2 - q| \vee |q|$ for every $B$, it follows from the Dominated Convergence Theorem 
\begin{align}
1 + \frac{\Esp_{\hat{q}} (o_B(|\hat{q} - q|) \indc_{\tilde{s} \geq 0})}{\overline{M} \vartheta} = \ens 1 + \frac{o_B(1)}{\overline{M} \vartheta }
\label{apply.dom.conv.th}
\end{align}
On the other hand, Lemma~\ref{LemmaGapQtl} provides
\begin{align}
\Esp (\hat{s} - s)^2 = \frac{\Esp (\hat{q}-q)^2}{n^2} \leq \frac{C_{1, P} + \alpha C_{2, P} / B }{n^2 \alpha B} \leq \frac{C_{P^b}}{n^2 \alpha B} \ens . 
\end{align}
so that an upper bound for the Type-II error is given by
\begin{align}
& \left(\frac{3}{2} + o_n(1)\right)\exp\left( - \frac{n s^2}{2 \vartheta^2 + (2/3) \overline{M} \vartheta s}  \right) \left\{ 1 +  \frac{3 C_{P^b}}{2 \overline{M} \vartheta \sqrt{\alpha B}} +  \frac{o_B(B^{-1/2})}{\overline{M}^2 \vartheta^2}  \right\} \ens . \label{final.upper.bound} 
\end{align}
which one rewrites as
\begin{align}
 \exp\left( - \frac{n \left(L - \displaystyle\frac{q}{n L} \right)^2}{2 m_P^2 + C m_P M^{1/2}(L^2 - q/n) }    \right) f_1(B, M, L) \ens ,\notag 
\end{align}
where
\begin{align}
f_1(B, M, L) & = (3/2 + o_n(1)) \left(1 + \frac{C_{P^b}}{C^{'} L^2 M^{1/2} m_P \sqrt{\alpha B}} + \frac{o_B(B^{-1/2})}{C^{''} L^4 M m_P^2}\right) \notag \ens ,
\end{align}
and $C = (2/3)(4+\sqrt{2})$, $C^{'} = 2 (4 + \sqrt{2}) / 3$ and $C^{''} = (4 + \sqrt{2})^2$.
\end{enumerate}

Theorem~\ref{PropTypeIIerr} is proved.

\subsection{Auxilary results}
\begin{lemma}{(Bennett's inequality, Theorem 2.9 in \cite{Book_ConcIneq})}
\label{LemmaBennett}
Let $\xi_1, \ldots, \xi_n$ i.i.d. zero-mean variables bounded by $c$ and of variance $\nu^2$.\\
Then, for any $\epsilon > 0$
\begin{align}
\Prb\left(\sum_{i=1}^n \xi_i \geq \epsilon \right) \leq \exp\left( - \frac{\epsilon^2}{2 n \nu^2 + 2 c \nu \epsilon / 3}  \right) \ens .
\end{align}
\end{lemma}
\begin{lemma}{(Theorem 2.9. in \cite{BoucheronOrderStat})}
\label{LemmaGapQtl}
Assume $\alpha < 1/2$. Then,
\begin{equation}
\mathrm{Var}(\hat{q}_{\alpha, n}) \leq \frac{C_{P^b}}{\alpha B} \ens ,
\end{equation}
where $C_{P^b}$ only depends on the bootstrap distribution (of $[n \hLMMD^2]^b_{fast}$).
\end{lemma}

\begin{lemma}
\label{LemmaBound}
If $Y \in \hsp_0 \subseteq \hsp$ $P$-almost surely and $\sup_{x, y \in \hsp_0} |\bar{k}(x, y)| = M$, then for any $y \in \hsp_0$ 
\begin{align}
|f(y)| \leq \overline{M} := (4 + \sqrt{2})  L \sqrt{M} \ens .
\end{align}

\end{lemma}

\begin{proof}
\begin{align*}
|f(y)| = & \ens \left|  \langle  \bar{\mu}_P - N o \TT(\theta_0), D_{\theta_0} (N o \TT) (\Psi(y)) - \bar{\phi}(y) + \bar{\mu}_P\rangle   \right| \notag \\
			 \leq & L  \left\Vert   D_{\theta_0} (N o \TT) (\Psi(y)) - \bar{\phi}(y) + \bar{\mu}_P \right\Vert    \notag \\
			 \leq & \ens L \left( \left[ \lim_{t \to 0} \Esp_{Z, Z^{'} \sim N o \TT(\theta_0 + t \Psi(y))} \bar{k}(Z, Z) + \bar{k}(Z^{'}, Z^{'}) - 2 \bar{k}(Z, Z^{'}) \right]^{1/2} \right. \\
			 & \qquad \left. + \sqrt{\bar{k}(y, y)} + \sqrt{\Esp \bar{k}(Y, Y^{'})}  \right)   \notag \\
			 \leq & \ens  L \left( \sqrt{M} + \sqrt{M} + \sqrt{2 M} + \sqrt{M} + \sqrt{M} \right)   \notag \\
			 \leq & \ens L (4 + \sqrt{2}) \sqrt{M} := \overline{M}   \ens . \notag
\end{align*}
\end{proof}

\begin{lemma}
\label{LemmaVariance}
\begin{align}
\nu^2 \leq \vartheta^2 := L^2 m_P^{(2)}  \ens .
\end{align}
\end{lemma}

\begin{proof}
\begin{align*}
\nu^2 := \mathrm{Var}(f(Y)) = & \ens \Esp f^2(Y_i) \notag \\
	= & \Esp \langle  \bar{\mu}_P - N o \TT(\theta_0), D_{\theta_0} (N o \TT) (\Psi(Y)) - \bar{\phi}(Y) + \bar{\mu}_P\rangle ^2 \\   
	\leq & \ens L^2 \underbrace{\Esp \left\Vert D_{\theta_0} (N o \TT) (\Psi(Y)) - \bar{\phi}(Y) + \bar{\mu}_P  \right\Vert^2}_{:= m_P^2} \ens . \notag
\end{align*}
\end{proof}

\begin{lemma}
\label{LemmaTaylorLag}
Let $h$ be defined as in \eqref{BennettBound}. Then,
\begin{align}
h(\hat{s}) \leq \exp\left( - \frac{n s^2}{2 \vartheta^2 + (2/3) \overline{M} \vartheta  s}  \right) \left\{1 + \frac{3 n }{2 \overline{M} \vartheta} \exp\left( \frac{3 |\tilde{q} - q|}{2 \overline{M} \vartheta }  \right) \indc_{\tilde{s} \geq 0} |\hat{s} - s| \right\} \ens , 
\end{align}
where
\begin{align*}
s = L^2 - \frc{q}{n}  \ens , \quad \tilde{s} = L^2 - \frc{\tilde{q}}{n}  \ens , \quad \tilde{q} \in (q \wedge \hat{q}, q \vee \hat{q}) \ens .
\end{align*}
\end{lemma}
\begin{proof}
A Taylor-Lagrange expansion of order 1 can be derived for $h(\hat{s})$ since the derivative of $h$ 
\begin{align}
h'(x) = - \frc{(2/3) n \overline{M} \vartheta   x^2 + 4 n \vartheta^2 x}{(2 \vartheta^2 + (2/3) \overline{M} \vartheta  x )^2} \exp\left( - \frc{n x^2}{2 \vartheta^2 + (2/3) \overline{M} \vartheta x}  \right) \indc_{x \geq 0} \notag \ens ,
\end{align}
is well defined for every $x \in \mathbb{R}$ (in particular, the left-side and right-ride derivatives at $x=0$ coincide).\\
Therefore $h(\hat{s})$ equals 
\begin{align}
h(s) + h^{'}(\tilde{s}) & (\hat{s} - s)  \notag \\
& = \exp\left( - \frac{n s^2}{2 \vartheta^2 + (2/3) \overline{M} \vartheta  s}  \right) \left[ 1 + \exp\left( g(s) - g(\tilde{s})  \right) g'(\tilde{s}) \indc_{ \tilde{s} \geq 0} (\hat{s} - s)     \right]  
 \ens ,
\end{align}
where 
\begin{align}
s = L^2 - \frc{q}{n} \ens , \quad \tilde{s} = L^2 - \frc{\tilde{q}}{n}  \ens , \quad \tilde{q} \in (q \wedge \hat{q}, q \vee \hat{q}) \ens , \notag
\end{align}
and $s \geq 0$ because of the assumption $n > q L^{-2}$.\\

For every $x, y > 0$, $|g'(x)| \leq 3 n / (2 \overline{M} \vartheta)$ and then $|g(x) - g(y)| \leq 3 n |x - y| / (2 \overline{M} \vartheta )$. It follows
\begin{align*}
|g'(\tilde{s})| \leq \frac{3 n }{2 \overline{M} \vartheta } \ens ,
\end{align*}
and
\begin{align}
\exp\left( g(s) - g(\tilde{s})  \right) \leq \exp\left( \frac{3 n}{2 \overline{M} \vartheta } |s - \tilde{s}|  \right) =  \exp\left( \frac{3 |\tilde{q} - q^{'}|}{2 \overline{M} \vartheta }  \right)  \ens . \notag
\end{align}

Lemma~\ref{LemmaTaylorLag} is proved.
\end{proof}

\end{document}